\documentclass[11pt,reqno]{amsart}

\usepackage[utf8]{inputenc}

\usepackage{amsmath,amsthm,amssymb}
\usepackage{amssymb}
\usepackage{mathrsfs}

\usepackage{graphics}
\usepackage{hyperref}
\usepackage[usenames, dvipsnames]{xcolor}
\usepackage{soul}

\usepackage{mathtools}
\mathtoolsset{showonlyrefs}

\usepackage[square,sort,comma,numbers]{natbib}

\usepackage{verbatim}
\usepackage{longtable}

\definecolor{darkblue}{rgb}{0.0,0.0,0.3}
\hypersetup{colorlinks,breaklinks,
  linkcolor=darkblue,urlcolor=darkblue,
anchorcolor=darkblue,citecolor=darkblue}

\theoremstyle{plain}
\newtheorem{theorem}{Theorem}[section]
\newtheorem*{theorem*}{Theorem}
\newtheorem{lemma}[theorem]{Lemma}
\newtheorem{proposition}[theorem]{Proposition}
\newtheorem*{proposition*}{Proposition}
\newtheorem{corollary}[theorem]{Corollary}
\newtheorem*{corollary*}{Corollary}

\theoremstyle{definition}
\newtheorem{remark}[theorem]{Remark}

\numberwithin{equation}{section}

\renewcommand{\Im}{\operatorname{Im}}
\renewcommand{\Re}{\operatorname{Re}}

\DeclareMathOperator*{\Res}{Res}

\DeclareRobustCommand{\mhl}[1]{%
  \ifmmode\text{\hl{$#1$}}\else\hl{#1}\fi
}

\title{Self-Correlations of Hurwitz Class Numbers}
\author{Alexander Walker}

\begin{document}

\begin{abstract}
The asymptotic study of class numbers of binary quadratic forms is a foundational problem in arithmetic statistics. Here, we investigate finer statistics of class numbers by studying their self-correlations under additive shifts. Specifically, we produce uniform asymptotics for the shifted convolution sum $\sum_{n < X} H(n) H(n+\ell)$ for fixed $\ell \in \mathbb{Z}$, in which $H(n)$ denotes the Hurwitz class number. 
\end{abstract}

\maketitle

\section{Introduction}

The study of class numbers of binary quadratic forms has a rich history, dating back to Lagrange and Gauss. In \emph{Disquitiones Arithmeticae}, Gauss made several conjectures about the distribution of class numbers, including the famous statement that the class number $h(-D)$ of binary quadratic forms of discriminant $-D$ should diverge to infinity as $D \to \infty$. Gauss' conjecture was established by Heilbronn~\cite{Heilbronn34}, with effective lower bounds first obtained through the combined work of~\cite{Goldfeld76} and~\cite{GrossZagier83}.

Moment estimates for class numbers have been studied by many authors, often using Dirichlet's class number to reduce the problem to estimates for families of quadratic Dirichlet $L$-functions at the special point $1$. For example, Wolke~\cite{Wolke72} proves that
\begin{align} \label{eq:class-number-moment}
	\sum_{n \leq X} \widetilde{h}(-n)^\alpha = c(\alpha) X^{1+\frac{\alpha}{2}} 
	+ O_\alpha\big(X^{1+\frac{\alpha}{2} - \frac{1}{4}}\big)
\end{align}
for fixed $\alpha > 0$, where $\widetilde{h}(-n)$ denotes the number of classes of \emph{primitive} binary quadratic forms of discriminant $-n$.  Later work of Granville and Soundararajan~\cite{GranvilleSoundararajan03} implies that the main term in~\eqref{eq:class-number-moment} holds with some uniform error for any $\alpha \ll \log X$.

In comparison, shifted convolution estimates for class numbers are far less understood. Recent work of Kumaraswamy~\cite{Kumaraswamy18} considers
\[
	D(X,\ell) := \sideset{}{^\flat}\sum_{n \leq X} h(-n) h(-n-\ell),
\]
in which $\sum^\flat$ denotes restriction to $n$ such that both $-n$ and $-n-\ell$ are fundamental discriminants, with neither congruent to $1 \bmod 8$. Kumaraswamy applies the circle method to prove that
\[
	D(X,\ell) = c_\ell X^{\frac{3}{2}}(X+\ell)^\frac{1}{2}
		+ O_\epsilon\big( X^{\frac{3}{2} - \frac{1}{30}} (X+ \ell)^{\frac{1}{2} + \frac{1}{60} + \epsilon}\big)
\]
for $\ell \geq 1$ and all $\epsilon > 0$, uniformly in $\ell$ (cf.~\cite[Theorem 1.1]{Kumaraswamy18}). For fixed $\ell$, this gives a power-saving of $O(X^{1/60-\epsilon})$ in the error term.

Unfortunately, the peculiar restriction to $n, n+\ell \not\equiv 1 \bmod 8$ in~\cite{Kumaraswamy18} is essential, as this work uses the identity
\begin{align} \label{eq:r3-h-identity}
	r_3(n)  = 12 \Big( 1 - \Big(\frac{-n}{2}\Big)\Big) h(-n),
\end{align}
(cf.~\cite[Proposition 5.3.10]{Cohen93})
to relate the class number to the Kronecker symbol and $r_3(n)$, the number of representations of $n$ as a sum of $3$ squares, which holds when $-n$ is a fundamental discriminant. Since $(\frac{-n}{2}) = 1$ for $n \equiv \pm 1 \bmod 8$, the identity~\eqref{eq:r3-h-identity} gives no information about $h(-n)$ on the residue class $-n \equiv 1 \bmod 8$. 

This article presents an alternative method for studying correlations of class numbers, via the spectral theory of automorphic forms. In this setting, it is convenient to consider a version of the class number $h(-n)$ called the Hurwitz class number $H(n)$, in which the classes containing a multiple of $x^2+y^2$ or $x^2-xy+y^2$ are weighted by $\frac{1}{2}$ and $\frac{1}{3}$, respectively. By convention, we set $H(0) = - \frac{1}{12}$.
Hurwitz class numbers feature, for one example, in Eichler--Selberg ``class number relation'' formulas, such as
\begin{align} \label{eq:class-number-relation}
	\sum_{m \in \mathbb{Z}} H(4n-m^2) = 2 \sigma_1(n) + \sum_{d \mid n} \min\big(d, \frac{n}{d}\big),
\end{align}
which appear in the work of Kronecker and Hurwitz. Here, $\sigma_\nu(n) = \sum_{d \mid n} d^\nu$.

More recently, Zagier~\cite{Zagier75} showed that Hurwitz class numbers arise as the coefficients of a mock modular form. Specifically, Zagier proved that
\begin{align} \label{eq:H-definition}
	\mathcal{H}(z) := 
		\sum_{n \geq 0} H(n) e(nz) + \frac{1}{8\pi \sqrt{y}}
		+ \sum_{n \geq 1} \frac{n \,\Gamma(-\tfrac{1}{2}, 4\pi n^2 y)}{4\sqrt{\pi}} e(-n^2 z)
\end{align}
defines a harmonic Maass form of weight $\frac{3}{2}$ on $\Gamma_0(4)$. Here, $z=x+iy$, $e(z) = e^{2\pi i n z}$, and $\Gamma(\beta, y)$ denotes the incomplete gamma function. In particular, one may study Hurwitz class numbers using automorphic forms.

In this article, we leverage the analytic theory of harmonic Maass forms and mock modular forms to study the shifted convolution Dirichlet series
\begin{align} \label{eq:D_h-definition}
	D_\ell(s) :=
		\sum_{n \geq 1} \frac{H(n) H(n+\ell)}{(n+\ell)^{s+\frac{1}{2}}},
\end{align}
where $\ell \geq 1$ is a fixed integer. We prove that $D_\ell(s)$ admits meromorphic continuation to $s \in \mathbb{C}$ and use this information to study the self-correlations of Hurwitz class numbers under additive shifts. Our main theorem is the following result.

\begin{theorem} \label{thm:intro-main-theorem}
Let $\sigma_\nu(m) = \sum_{d \mid m} d^\nu$ denote the sum-of-divisors function, with the convention that $\sigma_\nu(m) = 0$ for $m \notin \mathbb{Z}$. Fix $\ell \geq 1$ and let $\ell_o$ denote the odd part of $\ell$. Then, for all $\epsilon > 0$, we have
\begin{align*}
	&\sum_{n \leq X} H(n) H(n+\ell) \\[-.2em]
	& \qquad 	= \frac{\pi^2 X^2}{252\zeta(3)}
				\big(2\sigma_{-2}(\tfrac{\ell}{4}) - \sigma_{-2}(\tfrac{\ell}{2})  +\sigma_{-2}(\ell_o) \big) 
	+ O_\epsilon\big(X^{\frac{5}{3}+\epsilon} 
		 + X^{1+\epsilon} \ell \big).
\end{align*}
\end{theorem}

For $\ell \ll X^{2/3}$, this result achieves a uniform error of size $O_\epsilon(X^{5/3+\epsilon})$. For larger $\ell$, the error term depends on $\ell$ but remains non-trivial when $\ell \ll X^{1-\epsilon}$.

Since Hurwitz class numbers agree with the ordinary class numbers $h(-n)$ for $n$ not of the form $3m^2$ or $4m^2$, the rough upper bound $H(n) \ll n^{1/2+o(1)}$ (cf.~Lemma~\ref{lem:H-growth}) implies the following result as an immediate corollary.

\begin{corollary*} With notation as above, we have
\begin{align*}
	& \sum_{n \leq X} \! h(-n-\ell)h(-n) \\[-.2em]
	& \qquad = \frac{\pi^2 X^2}{252\zeta(3)}
				\big(2\sigma_{-2}(\tfrac{\ell}{4}) - \sigma_{-2}(\tfrac{\ell}{2})+\sigma_{-2}(\ell_o) \big) 
	+ O_\epsilon\big(X^{\frac{5}{3}+\epsilon} + X^{1+\epsilon} \ell^{1+\epsilon} \big).
\end{align*}
\end{corollary*}

The error bounds in Theorem~\ref{thm:intro-main-theorem} are of course not sharp. We conjecture that Theorem~\ref{thm:intro-main-theorem} should hold with a secondary main term and an error of size $O_\epsilon((X\ell)^{1+\epsilon})$; specifically, that
\begin{align} \label{eq:intro-conjecture}
	\sum_{n \leq X} H(n) H(n+\ell) 
	&	= \frac{\pi^2 X^2}{252\zeta(3)}
				\big(2\sigma_{-2}(\tfrac{\ell}{4}) - \sigma_{-2}(\tfrac{\ell}{2})  +\sigma_{-2}(\ell_o) \big) \\
	& \,
		- \frac{2X^{\frac{3}{2}}}{9\pi}
				\big(2 \sigma_{-1}(\tfrac{\ell}{4}) - \sigma_{-1}(\tfrac{\ell}{2}) + \sigma_{-1}(\ell_o) \big)
		+ O_\epsilon\big((X\ell)^{1+\epsilon} \big).
\end{align}
To support this conjecture, we show (cf.~Remark~\ref{rem:smooth-result}) that~\eqref{eq:intro-conjecture} holds when the cutoff $n \leq X$ is replaced by a certain class of truncations with smoothing.

\section*{Paper Methodology and Outline}

To produce shifted convolution estimates that treat all congruence classes equally, we abandon~\eqref{eq:r3-h-identity} in favor of the generating function $\mathcal{H}(z)$ from~\eqref{eq:H-definition}. In particular, we treat shifted convolutions involving weak harmonic Maass forms instead of ordinary modular forms. We also depart from~\cite{Kumaraswamy18} in that we treat shifted convolutions using the spectral theory of automorphic forms, as opposed to the circle method.

Following some background material on harmonic weak Maass forms and mock modular forms in section~\S{\ref{sec:harmonic-Maass-background}}, we relate the Dirichlet series $D_\ell(s)$ defined in~\eqref{eq:D_h-definition} to the Petersson inner product $\langle y^{3/2} \vert \mathcal{H} \vert^2, P_\ell(\cdot, \overline{s}) \rangle$, in which $P_\ell(z,s)$ is a particular Poincar\'e series.

We obtain a meromorphic continuation for $D_\ell(s)$ by first producing a meromorphic continuation of $\langle y^{3/2} \vert \mathcal{H} \vert^2, P_\ell(\cdot, \overline{s}) \rangle$. This task is complicated by the fact that $F(z):=y^{3/2} \vert \mathcal{H}(z) \vert^2$ is not square-integrable. To address this, we show in section~\S{\ref{sec:automorphic-regularization}} that $F(z)$ may be written in the form $\mathcal{V}(z) + \mathcal{E}(z)$, in which $\mathcal{V} \in L^2$ and $\mathcal{E}$ is an explicit function involving Eisenstein series and the Jacobi theta function.

The meromorphic continuations of $\langle \mathcal{E}, P_\ell(\cdot,\overline{s}) \rangle$ and $\langle \mathcal{V}, P_\ell(\cdot,\overline{s}) \rangle$  are then computed in sections~\S{\ref{sec:regularization-inner-products}} and~\S{\ref{sec:spectral-expansion}}, respectively. While $\langle \mathcal{E}, P_\ell(\cdot,\overline{s}) \rangle$ can be understood directly, the meromorphic continuation of $\langle \mathcal{V}, P_\ell(\cdot,\overline{s}) \rangle$ is accomplished through spectral expansion of the Poincar\'e series.

The methods described up to this point apply more generally. To illustrate this, the major results of sections~\S{\ref{sec:unfolding}-\ref{sec:spectral-expansion}} are presented with $\mathcal{H}$ replaced by a generic weak harmonic Maass form ``of polynomial growth'' (cf.~\S{\ref{sec:harmonic-Maass-polynomial-growth}}). Our first significant specialization to $\mathcal{H}(z)$ occurs in~\S{\ref{sec:spectral-expansion}}, where we leverage the fact that the contribution from the non-holomorphic part of $\mathcal{H}(z)$ is unusually simple (cf. Remark~\ref{rem:H-specialization}) to more easily classify the poles and residues of $D_\ell(s)$ in the right half-plane $\Re s > \frac{1}{2}$.

Our main application, Theorem~\ref{thm:intro-main-theorem}, also requires uniform bounds for the growth of $D_\ell(s)$ in vertical strips. In section~\S{\ref{sec:D_h-growth-bounds}}, we address various elementary terms to reduce this problem to growth estimates for $\langle \mathcal{V}, P_\ell(\cdot, \overline{s}) \rangle$.

The spectral expansion of $P_\ell(z,s)$ gives a decomposition $\langle \mathcal{V}, P_\ell(\cdot, \overline{s}) \rangle = \Sigma_{\mathrm{disc}}(s) + \Sigma_{\mathrm{cont}}(s)$ corresponding to contributions from the discrete and continuous spectra of the hyperbolic Laplacian. While $\Sigma_{\mathrm{cont}}$ is readily handled, the problem of bounding $\Sigma_{\mathrm{disc}}(s)$ with respect to $\vert \Im s \vert$ is particularly complicated and represents the central difficulty of this work.

Ultimately, our bounds for $\Sigma_{\mathrm{disc}}$ rely on decay estimates for triple inner products of the form $\langle y^{3/2} \vert \mathcal{H} \vert^2, \mu_j \rangle$, in which $\mu_j(z)$ runs through an orthonormal basis for Hecke--Maass cusp forms on $\Gamma_0(4)$. Similar inner products, of the form $\langle y^k \phi_1 \overline{\phi_2}, \mu_j  \rangle$ (with $\phi_1,\phi_2$ automorphic forms of weight $k$) have been studied in numerous works, and we mention a few:
\begin{enumerate}
	\item[a.] $\phi_1, \phi_2$ weight $k \in \mathbb{Z}$ holomorphic cusp forms on $\Gamma_0(N)$, by~\cite{Good81};
	\item[b.] $\phi_1, \phi_2$ weight $0$ Eisenstein series on $\Gamma_0(1)$, by~\cite{VinogradovTakhtadzhyan87};
	\item[c.] $\phi_1\overline{\phi_2}$ replaced by any polynomial in Maass cusp forms, by~\cite{Sarnak94};
	\item[d.] $\phi_1, \phi_2$ weight $0$ Maass cusp forms on $\Gamma_0(1)$, by~\cite{Jutila96,Jutila97};
	\item[e.] $\phi_1,\phi_2$ weight $k \in \frac{1}{2} \mathbb{Z}$ modular forms on $\Gamma_0(N)$, by~\cite{Kiral15}.
\end{enumerate}
Of these prior works, (a) and (b) use the Rankin--Selberg method directly, (c) and (e) use the automorphic kernel, and (d) uses a modified Rankin--Selberg method that introduces an auxiliary Eisenstein series for the express purpose of unfolding.

Our treatment of $\langle y^{3/2} \vert \mathcal{H} \vert^2, \mu_j \rangle$ appears in section~\S{\ref{sec:jutila-triple-inner-products}}. More generally, this section produces bounds for triple inner products of the form $\langle y^{k} \vert f \vert^2, \mu_j \rangle$, where $f$ is a harmonic Maass form of polynomial growth of weight $k \in \frac{1}{2} + \mathbb{Z}$.
In particular, we prove the following result:

\begin{theorem} \label{thm:intro-inner-product-bound}
Let $f$ be a harmonic Maass form of polynomial growth of weight $k \in \frac{1}{2} + \mathbb{Z}$ and level $N$. Let $\mu$ be an $L^2$-normalized Hecke--Maass cusp form of weight $0$ on $\Gamma_0(N)$, with spectral type $t \in \mathbb{R}$. For all $\epsilon > 0$, we have
\[
	\langle y^k \vert f \vert^2, \mu \rangle 
		\ll_{N,\epsilon}
		\big( \vert t \vert^{2k-1} + \vert t \vert^{3-2k} \big) \vert t \vert^\epsilon
		e^{-\frac{\pi}{2} \vert t \vert}.
\]
\end{theorem}
We remark that the space of harmonic Maass forms of polynomial growth includes $M_k(\Gamma_0(N))$, the space of modular forms. In this setting, Theorem~\ref{thm:intro-inner-product-bound} can be used to improve the spectral dependence in certain results of~\cite{Kiral15}. (In particular, see~\cite[Proposition 14]{Kiral15}.)

Our proof of Theorem~\ref{thm:intro-inner-product-bound} draws heavy inspiration from~\cite{Jutila96,Jutila97}, though our work is more complicated in several respects, such as the change from $\Gamma_0(1)$ to $\Gamma_0(N)$, the change in Whittaker functions (from $K$-Bessel functions to incomplete gamma functions), the generalization to half-integral weight, and the introduction of terms related to the fact that $f$ need not be cuspidal.
We also depart from Jutila by considering individual inner products instead of spectral large sieve inequalities. We suspect that a spectral large sieve inequality would not improve Theorem~\ref{thm:intro-main-theorem}.

In section~\S{\ref{sec:D_h-growth-II}}, we apply these triple product estimates to complete our quantification of the growth of $D_\ell(s)$. At this point, our main result follows from a version of Perron's formula with truncation, as presented in~\S{\ref{sec:perron}}.

\section*{Acknowledgments}

The author was supported by the Additional Funding Programme for Mathematical Sciences, delivered by EPSRC (EP/V521917/1) and the Heilbronn Institute for Mathematical Research. This work was also supported by the Swedish Research Council under grant no.\ 2016-06596 while the author was in residence at Institut Mittag-Leffler in Djursholm, Sweden.


\section{Harmonic Weak Maass Forms and Mock Modular Forms} \label{sec:harmonic-Maass-background}

The theory of harmonic Maass forms was introduced by Bruinier and Funke in the context of geometric theta lifts~\cite{BruinierFunke04}. This section reviews the basic definitions of harmonic Maass forms and mock modular forms. A good reference for background material is~\cite[\S{4}]{BFOR19}.

A weak Maass form of weight $k$ on a congruence subgroup $\Gamma \subset \mathrm{SL}_2(\mathbb{Z})$ is a smooth function $f: \mathfrak{h} \to \mathbb{C}$ which transforms like a modular form of weight $k$, is an eigenfunction of the weight $k$ Laplacian
\[
	\Delta_k := - y^2 \Big( \frac{\partial^2}{\partial x^2} + \frac{\partial^2}{\partial y^2}\Big) 
		+ i k y \Big( \frac{\partial}{\partial x} + i \frac{\partial}{\partial y} \Big),
\]
and has at most linear exponential growth at cusps. 

If $\Delta_k f = 0$, then $f$ is called a harmonic (weak) Maass form of manageable growth. Let $H_k^!(\Gamma)$ denote the space of weight $k$ harmonic Maass forms of manageable growth on $\Gamma$. 
If $\Gamma = \Gamma_0(N)$ or $\Gamma_1(N)$, then any $f(x+iy) \in H_k^!(\Gamma)$ admits a Fourier expansion at $\infty$ of the form
\begin{align} \label{eq:harmonic-Fourier-expansion}
f(z) = \sum_{n \geq n^+} &  c^+(n) e(nz) \\[-.5em]
& + c^-(0)y^{1-k} + \sum_{\substack{n \geq n^- \\ n \neq 0 }} c^-(n) \Gamma(1-k,4\pi n y) e(-nz)
\end{align}
(cf.~\cite[Lemma 4.3]{BFOR19}), where $\Gamma(\beta,y)$ denotes the incomplete gamma function
\[\Gamma(\beta,y) := \int_y^\infty t^{\beta} e^{-t} \frac{dt}{t}.\]
In the case $k=1$, the term $c^-(0)y^{1-k}$ is replaced with $c^-(0)\log y$. The two lines~\eqref{eq:harmonic-Fourier-expansion} in the Fourier expansion of $f(z)$ are called the holomorphic part and the non-holomorphic part, respectively. Any function which arises as the holomorphic part of a harmonic Maass form of manageable growth is called a mock modular form.

Fourier expansions of analogous shape exist for each cusp of $\Gamma$. To describe this precisely, we assume henceforth that $k \in \frac{1}{2}\mathbb{Z}$ and $\Gamma \subset \Gamma_0(4)$ and restrict to 
Maass forms with the theta multiplier system $\upsilon_\theta$. That is, where $\theta(z) := \sum_{n \in \mathbb{Z}} e(n^2 z)$ denotes the Jacobi theta function, we assume that
\begin{align*}
	f(\gamma z) &= \Big(\frac{\theta(\gamma z)}{\theta(z)}\Big)^{2k} f(z) 
\end{align*}
for all $\gamma = (\begin{smallmatrix} a & b \\ c& d \end{smallmatrix}) \in \Gamma$. This may also be written $f(\gamma z) = \upsilon_\theta(\gamma)^{2k} (cz+d)^k f(z)$, in which $\upsilon_\theta$ is defined by $\upsilon_\theta((\begin{smallmatrix} a & b \\ c & d \end{smallmatrix})) := \epsilon_d^{-1}(\frac{c}{d})$, where $\epsilon_d =1$ for $d \equiv 1 \bmod 4$ and $\epsilon_d = i$ for $d \equiv 3 \bmod 4$. For $\gamma  = (\begin{smallmatrix} a & b \\ c& d \end{smallmatrix}) \in \mathrm{GL}(2,\mathbb{R})$ with $\det \gamma > 0$, we define the weight $k$ slash operator by
\[
	f\vert_{\gamma}(z) := \Big(\frac{\theta(\gamma z)}{\theta(z)}\Big)^{-2k} f(\gamma z).
\]

Finally, for each cusp $\mathfrak{a}$ of $\Gamma$, let $\Gamma_\mathfrak{a} = \langle \pm t_\mathfrak{a} \rangle \subset \Gamma$ denote the stabilizer of $\mathfrak{a}$.  Let $\sigma_{\mathfrak{a}}$ denote a scaling matrix for $\mathfrak{a}$, i.e.\ a matrix in $\mathrm{GL}(2,\mathbb{R})$ for which $t_\mathfrak{a} = \sigma_\mathfrak{a} (\begin{smallmatrix} 1 & 1 \\ 0 & 1 \end{smallmatrix}) \sigma_{\mathfrak{a}}^{-1}$. Define the cusp parameter $\varkappa_\mathfrak{a} \in [0,1)$ so that $e(\varkappa_\mathfrak{a}) = \upsilon_\theta(t_\mathfrak{a})$. If $\varkappa_\mathfrak{a} = 0$, then the cusp $\mathfrak{a}$ is called singular; otherwise, $\mathfrak{a}$ is called non-singular.

Given all this notation, $f(z)$ admits a Fourier expansion at each cusp $\mathfrak{a}$ of $\Gamma$, given by
\begin{align} \label{eq:harmonic-Fourier-expansion-general}
	& f_\mathfrak{a}(z) := f\vert_{\sigma_{\mathfrak{a}}}(z) 
		= \sum_{n \geq n^+}   c_\mathfrak{a}^+(n) e((n+\varkappa_\mathfrak{a})z) \\
	& \qquad \qquad + c_\mathfrak{a}^-(0)y^{1-k} 
		+ \sum_{\substack{n \geq n^- \\ n \neq \varkappa_\mathfrak{a} }}
		c_\mathfrak{a}^-(n) \Gamma(1-k,4\pi (n - \varkappa_\mathfrak{a}) y) e(-(n-\varkappa_\mathfrak{a})z),
\end{align}
where $c_\mathfrak{a}^-(0) y^{1-k}$ appears only when $\varkappa_\mathfrak{a} = 0$. When $k=1$ and $\varkappa_\mathfrak{a}=0$, we replace this term by $c_\mathfrak{a}^-(0) \log y$. Since we work most commonly with the Fourier expansion at $\mathfrak{a} = \infty$, we retain the shorthand $c^\pm(n) := c_\infty^\pm(n)$.

\subsection{The Shadow Operator \texorpdfstring{$\xi_k$}{xi}}

This section follows~\cite[\S{5.1}]{BFOR19}. Recall the Maass lowering operator $L_k$ defined by $L_k = -iy^2(\frac{\partial}{\partial x} + i \frac{\partial}{\partial y})$. We define as well the shadow operator $\xi_k = y^{k -2} \overline{L_k}$. By~\cite[Theorem 5.10]{BFOR19}, $\xi_{k}$ maps $H_{k}^!(\Gamma_0(N))$ surjectively to $M_{2-k}^!(\Gamma_0(N))$, the space of weakly holomorphic modular forms of weight $2-k$. This map is given by
\begin{align} \label{eq:shadow-definition}
	\xi_{k}(f(z)) = (1-k) \overline{c^-(0)} 
	- (4\pi)^{1-k} \sum_{\substack{n \geq n^- \\ n \neq 0}} \overline{c^-(n)} n^{1-k} e(nz).
\end{align}
The form $\xi_{k} f$ is called the shadow of $f$.

\subsection{Harmonic Maass Forms of Polynomial Growth} \label{sec:harmonic-Maass-polynomial-growth}

Generically, the coefficient series $\{c_\mathfrak{a}^\pm(n)\}$ grow super-polynomially as $n \to \infty$. In the remainder of this article, we restrict to the special case in which the coefficients are polynomially bounded in $n$. This is equivalent to the property that $f(z)$ have no poles at cusps, or that $n^\pm \pm \varkappa_\mathfrak{a} \geq 0$ in~\eqref{eq:harmonic-Fourier-expansion} for all $\mathfrak{a}$.

Let $H^\sharp_k(\Gamma_0(N))$ denote the subspace of $H_k^!(\Gamma_0(N))$ consisting of forms with at most polynomial growth at cusps. We remark that the space $H_k^\sharp$ features prominently in~\cite{ShankhadharSingh22}, where it serves as a natural setting to study $L$-functions attached to mock modular forms. Note that $H_k^\sharp$ is a subspace of the space of (not-necessarily-cuspidal) Maass wave forms of weight $k$.

The shadow operator maps $\xi_k : H^\sharp(\Gamma_0(N)) \to M_{2-k}(\Gamma_0(N))$. In particular, $\xi_k$ annihilates $H_k^\sharp(\Gamma_0(N))$ for $k> 2$. In other words, $H_k^\sharp = M_k$ for $k> 2$, so the space $H_k^\sharp$ is most interesting for $k \leq 2$.
 
Though exact growth rates for the coefficients $c_\mathfrak{a}^\pm(n)$ are not known, adequate on-average bounds are known from the Rankin--Selberg method as applied to Maass forms (including non-cuspidal Maass forms) in~\cite{Muller95}. Specializing to the case of harmonic Maass forms and translating notation, we present the following result.

\begin{lemma}[cf.~Theorem 5.2,~\cite{Muller95}] \label{lem:Muller-coefficient-averages}
Fix $f(z) \in H_k^\sharp(\Gamma_0(N))$ with $k \in \frac{1}{2} \mathbb{Z}$ and $k \neq 1$. If $f$ has Fourier expansion~\eqref{eq:harmonic-Fourier-expansion-general}, then
\[\,\,\, 
	\sum_{n \leq X} \frac{\vert c_\mathfrak{a}^{\pm}(n) \vert^2}{(n \pm \varkappa_\mathfrak{a})^{k-1}}
		= \begin{cases}
			c_{f,\mathfrak{a}}^\pm X + O_f\big(X^{\frac{3}{5}} \log X\big),  
				\qquad \qquad \qquad \quad f \text{ cuspidal}, \\
			c_{f,\mathfrak{a}}^\pm X^{1+ \vert k -1 \vert}
		+ O_f\big( X^{1 + \vert k-1 \vert - \frac{2 + 4 \vert k-1 \vert}{5+ 8 \vert k -1 \vert}} \log X\big),
			\quad \, \text{else},
			\end{cases}
\]
for some constants $c_{f,\mathfrak{a}}^\pm$.
\end{lemma}


\section{Shifted Convolutions via Inner Products}  \label{sec:unfolding}

In this section, we show that shifted convolution Dirichlet series of the form~\eqref{eq:D_h-definition} can be recognized in terms of Petersson inner products. To begin, we treat a generic form $f(z) \in H_k^\sharp(\Gamma_0(N))$ with Fourier expansion~\eqref{eq:harmonic-Fourier-expansion}. We define the $\ell$-th Poincar\'e series $P_\ell(z,s)$ of weight $0$ on $\Gamma_0(N)$ by
\[
	P_\ell(z,s) := \sum_{\gamma \in \Gamma_\infty \backslash \Gamma_0(N)} \Im(\gamma z)^s e(\ell \gamma z).
\]
For $s$ with $\Re s$ sufficiently large, the Rankin--Selberg unfolding method gives
\begin{align} \label{eq:Poincare-unfolding}
\langle y^{k} \vert f \vert^2, P_\ell(\cdot, \overline{s}) \rangle
	&= \int_0^\infty \int_0^1 
		y^{s+k} \vert f(z) \vert^2 \overline{e(\ell z)} \frac{dxdy}{y^2} \\
	&= \sum_{n_1 = n_2 +\ell} \int_0^\infty y^{s+k-1} c(n_1,y) \overline{c(n_2,y)} e^{-2\pi h y} \frac{dy}{y},
\end{align}
in which $c(n,y)$ denotes the $n$-th Fourier coefficient of $f(z)$ at the cusp $\mathfrak{a} = \infty$. In other words, $c(n,y) = c^+(n) e^{-2\pi n y}$ for $n \geq 1$, $c(0,y) = c^+(0) + c^-(0) y^{1-k}$, and $c(n,y) = c^-(-n) \Gamma(1-k,-4\pi n y) e^{-2\pi n y}$ for $n \leq -1$.

The contribution of $n_1, n_2 > 0$ to the inner product is a standard shifted convolution Dirichlet series:
\begin{align*}
I_\ell^+(s) &:=	\!\! \sum_{n_1=n_2+\ell} \!\!
	c^+(n_1) \overline{c^+(n_2)} \int_0^\infty y^{s+k-1} e^{-2\pi (n_1+n_2+\ell) y} \frac{dy}{y} \\
	&\phantom{:}= \frac{\Gamma(s+k-1)}{(4\pi)^{s+k-1}}
		\sum_{n_2 =1}^\infty \frac{c^+(n_2+\ell)\overline{c^+(n_2)}}{(n_2+\ell)^{s+k-1}}.
\end{align*}
Since $f \in H_k^\sharp$, the Dirichlet series in the line above converges absolutely in some right half-plane. More precisely, Lemma~\ref{lem:Muller-coefficient-averages} gives convergence in $\Re s > 1 + \vert k-1 \vert$, extending to $\Re s > 1$ in the cuspidal case.

The net contributions from $(n_1,n_2) = (\ell,0)$ and $(0,-\ell)$ total
\begin{align} \label{eq:I_h^0-definition}
	I_\ell^0(s) :=& \frac{c^+(\ell) \overline{c^-(0)} \Gamma(s)}{(4\pi \ell)^s}
		+ \frac{c^+(\ell) \overline{c^+(0)} \Gamma(s+k-1)}{(4\pi \ell)^{s+k-1}} \\
	&\quad + \frac{c^-(0) \overline{c^-(\ell)}\Gamma(s-k+1)}{(4\pi \ell)^s s} 
		+ \frac{c^+(0) \overline{c^-(\ell)} \Gamma(s)}{(4\pi \ell)^{s+k-1} (s+k-1)}.
\end{align}
The function $I_\ell^0(s)$ is meromorphic in $s \in \mathbb{C}$ and analytic in $\Re s > \vert k-1\vert$. Note that $I_\ell^0(s)$ vanishes when $f$ is cuspidal.

There is another finite collection of ``cross terms'' when $n_1 > 0$ and $n_2 < 0$, which contributes
\begin{align} \nonumber
	I_\ell^\times(s) & := \sum_{m=1}^{\ell-1} c^+(\ell-m) \overline{c^-(m)}
		\int_0^\infty y^{s+k-1} e^{-4\pi (\ell-m) y} \Gamma(1-k, 4\pi m y) \frac{dy}{y} \\ \label{eq:I_h^x-definition}
	& \phantom{:}= \frac{\Gamma(s)}{s+k-1} \sum_{m=1}^{\ell-1}
		\frac{c^+(\ell-m) \overline{c^-(m)}}{(4\pi m)^{s+k-1}}
		{}_2F_1\Big( \begin{matrix} s, s+k-1 \\ s+k \end{matrix} \Big\vert 1-\frac{\ell}{m} \Big),
\end{align}
in which we've evaluated the integral via~\cite[6.455(1)]{GradshteynRyzhik07}. The function $I_\ell^\times(s)$ is analytic in the right half-plane $\Re s > \max(0, 1-k)$ and has an obvious meromorphic continuation to $s \in \mathbb{C}$.

Lastly, we record the contribution of $n_1,n_2 < 0$, which can be written
\begin{align} \label{eq:I-minus-definition}
	I_\ell^-(s) &:=  
	\sum_{n=1}^\infty \frac{c^-(n) \overline{c^-(n+\ell)}}{(4\pi)^{s+k-1}} G_k(s, n, n+\ell); \\
	& G_k(s,n,n+\ell):= \int_0^\infty y^{s+k-1} \Gamma(1-k, n y)  \Gamma(1-k, (n+\ell) y) e^{n y} \frac{dy}{y}.
\end{align}
The two asymptotic expressions $\Gamma(\beta,y) = \Gamma(\beta) - y^\beta/\beta + O_\beta(y^{\beta +1})$ as $y \to 0$ and $\Gamma(\beta, y) = e^{-y} y^{\beta -1} (1+O_\beta(y^{-1}))$ as $y \to \infty$ imply that $G_k(s,n,n+\ell)$ converges absolutely when $\Re s > \vert k -1 \vert$. In this region, $G_k(s,n,n+\ell) \ll G_k(\Re s, n,n)  \ll_{\Re s } n^{-\Re s -k+1}$ by change of variable. Thus $I_\ell^-(s)$ converges to an analytic function in $\Re s > 1 + \vert k - 1 \vert$, extending to the domain $\Re s > 1$ in the cuspidal case.

We conclude that the unfolding procedure is valid in $\Re s > 1 + \vert k -1 \vert$, and that in this region we have the decomposition
\begin{align} \label{eq:inner-product-decomposition}
	\langle y^k \vert f \vert^2, P_\ell(\cdot, \overline{s}) \rangle 
	= I_\ell^+(s) + I_\ell^0(s) + I_\ell^\times(s) + I_\ell^-(s).
\end{align}

\subsection{Application to \texorpdfstring{$\mathcal{H}(z)$}{H(z)}} \label{sec:H-decomposition}

The formulas in this section simplify considerably for the specific form $\mathcal{H} \in H_{3/2}^\sharp(\Gamma_0(4))$ defined in~\eqref{eq:H-definition}. Recall from~\eqref{eq:D_h-definition} the definition of the shifted convolution Dirichlet series
\begin{align} \label{eq:D_h-definition-redux}
	D_\ell(s) :=
		\sum_{n \geq 1} \frac{H(n) H(n+\ell)}{(n+\ell)^{s+\frac{1}{2}}},
\end{align}
which converges absolutely in $\Re s > \frac{3}{2}$. Simplifying the various terms at right in~\eqref{eq:inner-product-decomposition} produces the formula
\begin{align} \label{eq:Hurwitz-inner-product-decomposition}
\nonumber
\langle y^{\frac{3}{2}} \vert \mathcal{H} \vert^2, P_\ell(\cdot, \overline{s}) \rangle 
	&= \frac{\Gamma(s+\tfrac{1}{2})}{(4\pi)^{s+\frac{1}{2}}}
		D_\ell(s)
		 - \frac{H(\ell)\Gamma(s-\frac{1}{2})}{12 (4\pi \ell)^{s+\frac{1}{2}}} 
	   - \frac{H(\ell)\Gamma(s)}{8\pi (4\pi \ell)^s} \\
	& \quad
		- \frac{r_1(\ell) \Gamma(s-\frac{1}{2})}{128 \pi^{2} (4\pi \ell)^{s-\frac{1}{2}} s}
		- \frac{r_1(\ell) \Gamma(s)}{192 \pi (4\pi \ell)^{s} (s+\frac{1}{2})} \\
	& \quad
		+ \frac{\Gamma(s)}{s+\frac{1}{2}} 
		\sum_{m=1}^{\ell-1} \frac{H(\ell-m)  r_1(m)}{16 \pi (4\pi m)^{s}} 
		{}_2F_1\Big(\begin{matrix} s, s+\frac{1}{2} \\ s+\frac{3}{2} \end{matrix}
					\Big\vert 1-\frac{\ell}{m} \Big) \\
	& \quad + \sum_{\substack{m_1^2 - m_2^2 =\ell \\ m_1,m_2 \geq 1}}
		\frac{m_1 m_2}{4 (4\pi)^{s+\frac{3}{2}}} G_\frac{3}{2}(s,m_2^2, m_1^2).
\end{align}

We remark that the contribution of $I_\ell^-(s)$ in the fourth line of~\eqref{eq:Hurwitz-inner-product-decomposition} is a finite sum, since $m_1^2 - m_2^2 = \ell$ has finitely many solutions. This phenomenon generalizes to any $f \in H_{3/2}^\sharp(\Gamma_0(N))$ for which $M_{1/2}(\Gamma_0(N))$ is one-dimensional, since in that case $\xi_{3/2} f$ is necessarily a twisted theta function by~\cite[Theorem A]{SerreStark77}.
Thus, in departure from the general case, we conclude that $I_\ell^-(s)$ is analytic in $\Re s > \frac{1}{2}$.

Secondly, we remark that the contribution of $I_\ell^\times(s)$ bears some resemblance to one side of the Eichler--Selberg class number relation (cf.~\eqref{eq:class-number-relation}). More specifically,
\[
	\Res_{s=0} \frac{\Gamma(s)}{s+\frac{1}{2}} 
		\sum_{m=1}^{\ell-1} \frac{H(\ell-m)  r_1(m)}{16 \pi (4\pi m)^{s}} 
		{}_2F_1\Big(\begin{matrix} s, s+\frac{1}{2} \\ s+\frac{3}{2} \end{matrix}
		\Big\vert 1-\tfrac{\ell}{m} \Big) 
	=
	 \frac{1}{4\pi} \! \sum_{m^2 < \ell} H(\ell-m^2),
\]
which is essentially one of the sums described in~\cite[\S{10.3}]{BFOR19}. It would be interesting to know if the methods in this paper could be used to produce new class number relations.


\section{Automorphic Regularization} \label{sec:automorphic-regularization}

To produce a meromorphic continuation for the Dirichlet series $D_\ell(s)$, we first show that the inner product $\langle y^k \vert \mathcal{H} \vert^2, P_\ell(\cdot, \overline{s}) \rangle$ has a meromorphic continuation to a larger domain. This latter continuation involves the spectral decomposition of $P_\ell(z,s)$ with respect to the hyperbolic Laplacian and is complicated by the fact that $y^{3/2} \vert \mathcal{H}(z) \vert^2 \notin L^2(\Gamma_0(4) \backslash \mathfrak{h})$. To rectify this, we modify $y^{3/2} \vert \mathcal{H}(z) \vert^2$ by subtracting a linear combination of automorphic forms chosen to neutralize growth at the cusps of $\Gamma_0(N)$.

We define the weight $0$ Eisenstein series attached to cusp $\mathfrak{a}$ of $\Gamma_0(N)$ by
\[
	E_\mathfrak{a}(z,s) = \sum_{\gamma \in \Gamma_\mathfrak{a} \backslash \Gamma_0(N)}
		\Im(\sigma_\mathfrak{a}^{-1}\gamma z)^s.
\]
These Eisenstein series have Fourier expansion at the cusp $\mathfrak{b}$ of the form
\begin{align} \label{eq:Fourier-Eisenstein}
	E_\mathfrak{a}(\sigma_\mathfrak{b} z,w) 
		= \delta_{[\mathfrak{a}=\mathfrak{b}]} y^w 
			&+ \pi^\frac{1}{2} \frac{\Gamma(w-\frac{1}{2})}{\Gamma(w)} \varphi_{\mathfrak{ab}0}(w) y^{1-w} \\
		&+ \frac{2\pi^w y^\frac{1}{2}}{\Gamma(w)} 
		\sum_{n \neq 0} \varphi_{\mathfrak{ab}n}(w) \vert n \vert^{w-\frac{1}{2}} 
			K_{w-\frac{1}{2}}(2\pi \vert n \vert y) e(nx),
\end{align}
in which $\delta_{[\cdot]}$ denotes the Kronecker delta, $K_\nu(y)$ is the $K$-Bessel function, and the coefficients $\varphi_{\mathfrak{ab}n}(w)$ are described in~\cite{DeshouillersIwaniec82}, for example. 

As in the previous section, we first consider a general $f(z) \in H_k^\sharp(\Gamma_0(N))$, with $k \in \frac{1}{2} \mathbb{Z}$ ($k \neq 1$), specializing to $f = \mathcal{H}$ when convenient. Let $F_\mathfrak{a}(z) := y^k \vert f\vert_{\sigma_{\mathfrak{a}}}(z) \vert^2 = \Im(\sigma_{\mathfrak{a}} z)^k \vert f(\sigma_{\mathfrak{a}} z) \vert^2$. If $\varkappa_\mathfrak{a} \neq 0$, then $F_\mathfrak{a}(z)$ decays exponentially as $y \to \infty$ by the Fourier expansion~\eqref{eq:harmonic-Fourier-expansion-general}, and no regularization is required. Otherwise, when $\varkappa_\mathfrak{a} = 0$,~\eqref{eq:harmonic-Fourier-expansion-general} implies that
\begin{align} \label{eq:F-growth}
	F_\mathfrak{a}(z) = y^k \vert c_\mathfrak{a}^+(0) + c_\mathfrak{a}^-(0) y^{1-k} \vert^2 + O(y^{-M})
\end{align}
as $y \to \infty$ for all $M> 0$. It therefore suffices to regularize growth of sizes $y^k$, $y^1$, and $y^{2-k}$ at the singular cusps.

For $k>1$, the Eisenstein series $E_\mathfrak{a}(z,k)$ counteracts growth of size $y^k$ at $\mathfrak{a}$, while for $k<1$ we utilize $E_\mathfrak{a}(z,2-k)$ to address $y^{2-k}$. Unfortunately, this technique fails to regularize growth of size $y^1$, since $E_\mathfrak{a}(z,w)$ has a pole at $w=1$. In this case, we instead subtract a multiple of the constant term in the Laurent expansion of $E_\mathfrak{a}(z,w)$ at $w=1$, which we denote $\widetilde{E}_\mathfrak{a}(z,1)$, and which satisfies $\widetilde{E}_\mathfrak{a}(\sigma_{\mathfrak{b}} z,1) = \delta_{[\mathfrak{a} = \mathfrak{b}]} y - \pi \log y \Res_{w=1} \varphi_{\mathfrak{ab}0}(w) + \widetilde{c}_{\mathfrak{ab}}+O(e^{-2\pi y})$ for some constant $\widetilde{c}_{\mathfrak{ab}}$. Thus, for example,
\begin{align} 
	\begin{split} \label{eq:regularization}
		\mathcal{V}_f(z) := F(z) 
				&- \sum_{\mathfrak{a} : \varkappa_\mathfrak{a} = 0}
				\vert c_\mathfrak{a}^-(0) \vert^2 E_\mathfrak{a}(z,2-k) \\
				&- 2\sum_{\mathfrak{a} : \varkappa_\mathfrak{a} = 0}
				\Re ( c_\mathfrak{a}^+(0) \overline{c_\mathfrak{a}^-(0)})
				\widetilde{E}_\mathfrak{a}(z,1)
	\end{split}
\end{align}
satisfies $\mathcal{V}_f(\sigma_{\mathfrak{a}} z) = O(y^{k} + \log y)$ as $y \to \infty$ when $k<1$. If $k<\frac{1}{2}$, it follows that $\mathcal{V}_f \in L^2(\Gamma_0(N) \backslash \mathfrak{h})$. The case $k>\frac{3}{2}$ may be treated analogously. 

The situation is more complicated in weights $k=\frac{1}{2}$ and $k=\frac{3}{2}$, as here we must regularize terms of size $y^{1/2}$. The obvious choice for regularizing $y^{1/2}$ is to subtract a multiple of $E_\mathfrak{a}(z,\frac{1}{2})$, but this term equals $0$ since the completed Eisenstein series $\zeta^*(2w) E_\mathfrak{a}(z,w)$ is analytic at $w=\frac{1}{2}$. Likewise, it is not possible to regularize with a linear combination of terms of the form $\lim_{w\to \frac{1}{2}} \zeta^*(2w) E_\mathfrak{a}(z,w)$, as these grow as $y^{1/2} \log y$ near $\mathfrak{a}$.

In weight $k =\frac{3}{2}$, the growth of size $y^{1/2}$ comes from the non-holomorphic part (cf.~\eqref{eq:F-growth}). In particular, we can regularize all cusp growth of size $y^{1/2}$ simultaneously by subtracting an appropriate multiple of $y^{1/2} \vert \xi_{3/2} f \vert^2$. Specifically, we define
\begin{align} \label{eq:regularization-3/2}
\begin{split} 
	\mathcal{V}_f(z) := F(z) 
		& - \sum_{\mathfrak{a} : \varkappa_\mathfrak{a} = 0}
			\vert c_\mathfrak{a}^+(0) \vert^2 E_\mathfrak{a}(z,\tfrac{3}{2}) \\
		& - 2\sum_{\mathfrak{a} : \varkappa_\mathfrak{a} = 0}
			\Re ( c_\mathfrak{a}^+(0) \overline{c_\mathfrak{a}^-(0)})
				\widetilde{E}_\mathfrak{a}(z,1)
		- 4y^{\frac{1}{2}} \vert \xi_{\frac{3}{2}} f(z) \vert^2.
\end{split}
\end{align}
Then $\mathcal{V}_f(\sigma_{\mathfrak{a}} z) = O(\log y)$ at each cusp by~\eqref{eq:shadow-definition}, so $\mathcal{V}_f \in L^2(\Gamma_0(N) \backslash \mathfrak{h})$.

In weight $k=\frac{1}{2}$, we may likewise attempt to regularize by subtracting a function of the form $y^{1/2} \vert g(z) \vert^2$, where $g \in M_{1/2}(\Gamma_0(N))$. However, there is no guarantee that a modular form with compatible cusp growth need exist. If $f \in H_{1/2}^\sharp(\Gamma_0(N))$ is chosen, we may test for the existence of a compatible $g$ using the basis for $M_{1/2}(\Gamma_0(N))$ described in~\cite[Theorem A]{SerreStark77}. Since we do not require $k=\frac{1}{2}$ for our principal application, we leave the question of the existence of a compatible $g$ as an interesting open problem.

\subsection{Automorphic Regularization of \texorpdfstring{$\mathcal{H}(z)$}{H(z)}} 

In practical terms, the problem of regularizing $\mathcal{H}(z)$ reduces to the problem of computing the constant Fourier coefficients $c_\mathfrak{a}^\pm(0)$ at each singular cusp $\mathfrak{a}$ of $\Gamma_0(4)$. In this section, we determine these coefficients, as summarized in the following proposition.

\begin{proposition} \label{prop:H-cusp-behavior}
Let $\mathcal{H} \in H_{3/2}^\sharp(\Gamma_0(4))$ denote Zagier's non-holomorphic Eisenstein series from~\eqref{eq:H-definition}. The cusp $\mathfrak{a} = \frac{1}{2}$ is non-singular for $\upsilon_\theta$; for the other cusps, $\mathcal{H}(z)$ has a Fourier expansion of the form~\eqref{eq:harmonic-Fourier-expansion-general}, in which
\[
	c_\infty^+(0) = - \frac{1}{12}, \qquad
	c_\infty^-(0) = \frac{1}{8\pi}, \qquad
	c_0^+(0) = \frac{1}{24}, \qquad
	c_0^-(0) = -\frac{1}{8\pi}.
\]
Consequently, the function
\begin{align}
	\mathcal{V}_\mathcal{H}(z) &:= y^{\frac{3}{2}} \vert \mathcal{H}(z) \vert^2 
		 - \frac{1}{144} E_\infty(z,\tfrac{3}{2}) - \frac{1}{576} E_0(z,\tfrac{3}{2}) \\
		& \qquad + \frac{1}{48 \pi} \widetilde{E}_\infty(z,1) + \frac{1}{96 \pi} \widetilde{E}_0(z,1)
		- \frac{1}{64 \pi^2} y^{\frac{1}{2}} \vert \theta(z) \vert^2
\end{align}
lies in $L^2(\Gamma_0(4) \backslash \mathfrak{h})$.
\end{proposition}

\begin{proof}
To verify that $\mathfrak{a} = \frac{1}{2}$ is non-singular, we first note that $\Gamma_{1/2}$ is generated by $t_{1/2} = (\begin{smallmatrix} -1 & 1 \\ -4 & 3 \end{smallmatrix})$. Since $\upsilon_\theta(t_{1/2}) = i$, we have $\varkappa_{1/2} = \frac{1}{4} \neq 0$.

As for the singular cusps, we clearly have $c_\infty^+(0) = -\frac{1}{12}$ and $c_\infty^-(0) = \frac{1}{8\pi}$ from the Fourier expansion~\eqref{eq:H-definition}. 
To understand the behavior of $\mathcal{H}(z)$ near $\mathfrak{a} = 0$, we follow~\cite{HirzebruchZagier76} and relate $\mathcal{H}(z)$ to certain Eisenstein series of weight $\frac{3}{2}$. Specifically, we introduce the Eisenstein series
\[
	E_{\frac{3}{2},s}(z) 
		:= \mathop{\sum\sum}\limits_{\substack{m>0, n \in \mathbb{Z} \\ (m,2n)=1}}
				\frac{(\frac{n}{m}) \, \epsilon_m}{(mz+n)^{3/2} \vert m z + n \vert^{2s}},
\]
as well as a second Eisenstein series $F_{3/2,s}(z) := z^{-3/2} \vert z \vert^{-2s} E_{3/2,s}(-1/4z)$. Though $E_{3/2,s}$ converges only for $\Re s > \frac{1}{4}$,~\cite[Theorem 2]{HirzebruchZagier76} implies that $E_{3/2,s}$ and $F_{3/2,s}$ have meromorphic continuation to $s \in \mathbb{C}$ and that
\[
\mathcal{F}_s(z) := - \tfrac{1}{96}\big((1-i) E_{3/2,s}(z) - i F_{3/2,s}(z) \big)
\]
satisfies $\mathcal{F}_0(z) = \mathcal{H}(z)$.

To investigate $\mathcal{H}(z)$ near the cusp $0$, we compute a partial Fourier expansion of $\mathcal{F}_s \vert_{\sigma_0}(z)$, where $\sigma_0 = (\begin{smallmatrix} 0 & -1 \\ 4 & 0 \end{smallmatrix})$. The functional equation $\theta(-1/4z) = (-2iz)^{1/2} \theta(z)$ implies that the weight $k=\frac{3}{2}$ slash operator satisfies
\begin{align*}
	\mathcal{F}_s \vert_{\sigma_0}(z) 
		&= (-2iz)^{-\frac{3}{2}} \mathcal{F}_s(-\tfrac{1}{4z}) \\
		&=  - \tfrac{1}{96} (-2iz)^{-\frac{3}{2}} 
					\big((1-i) z^{\frac{3}{2}} \vert z \vert^{2s} F_{3/2,s}(z) 
						- i (-\tfrac{1}{4z})^{-\frac{3}{2}} \vert 4z \vert^{2s} E_{3/2,s}(z) \big) \\
		&= \frac{-i}{192} \vert z \vert^{2s} F_{3/2,s}(z) 
					+ \frac{1-i}{48} \vert z \vert^{2s} 2^{4s} E_{3/2,s}(z).
\end{align*}
Thus $\mathcal{H}\vert_{\sigma_0}(z)$ has a Fourier expansion which may be read from the Fourier coefficients of $E_{3/2,0}(z)$ and $F_{3/2,0}(z)$. Since we require only the constant Fourier coefficient of $\mathcal{H}\vert_{\sigma_0}$, it suffices to consider the constant Fourier coefficients of $E_{3/2,s}$ and $F_{3/2,s}$.

By~\cite[p.93-94]{HirzebruchZagier76}, the constant Fourier coefficient of $E_{3/2,s}(z)$ equals
\begin{align*}
	& \alpha_0(s,y) \sum_{m \text{ odd}} \frac{\epsilon_m \sum_{j(m)} (\frac{j}{m})}{m^{2s+\frac{3}{2}}} 
	= \alpha_0(s,y) \sum_{m \text{ odd}} \frac{\epsilon_{m^2} \varphi(m^2)}{(m^2)^{2s+\frac{3}{2}}}  \\
	& \qquad = \alpha_0(s,y) \sum_{m \text{ odd}} \frac{\varphi(m)}{m^{4s+2}}
	= \alpha_0(s,y) \frac{\zeta(4s+1)}{\zeta(4s+2)} \cdot \frac{1- 2^{-4s-1}}{1-2^{-4s-2}},
\end{align*}
in which $\varphi(m)$ denotes the totient function and $\alpha_0(s,y)$ is defined by
\[
	\alpha_0(s,y) = \int_{\Im w = y} w^{-\frac{3}{2}} \vert w \vert^{-2s} dw
		= - \frac{2(1+i)\sqrt{\pi} s \Gamma(2s+\frac{1}{2})}{\Gamma(2s+2)} y^{-\frac{1}{2}-2s}.
\]
By taking the limit as $s \to 0$, we conclude that the constant Fourier coefficient of $E_{3/2,0}(z)$ equals $-(2+2i)/\pi \cdot y^{-1/2}$.

Similarly, formulas on~\cite[p.94]{HirzebruchZagier76} show that $F_{3/2,s}$ has constant Fourier coefficient
\begin{align*}
	& 2^{2s+3} i +(1+i)\alpha_0(s,y) \sum_{m \text{ even}}
		\frac{m^{-1/2} \sum_{j(2m)} (\frac{m}{j}) e(\frac{j}{8})}{(m/2)^{2s+1}}.
\end{align*}
The sum $\sum_{j \bmod 2m} (\frac{m}{j})e(j/8)$ vanishes for $m$ even unless $m=2n^2$, in which case it equals $\sqrt{2} \varphi(2n^2)$. Thus the constant Fourier coefficient equals
\begin{align*}
	& 2^{2s+3} i + (1+i)\alpha_0(s,y) \sum_{n \geq 1} \frac{\varphi(2n^2)}{n^{4s+3}}
	= 2^{2s+3} i + (1+i) \alpha_0(s,y) \sum_{n \geq 1} \frac{\varphi(2n)}{n^{4s+2}} \\
	& \qquad
	= 2^{2s+3} i + (1+i) \alpha_0(s,y) 2^{4s+2}
		\frac{\zeta(4s+1)}{\zeta(4s+2)} \Big(1 - \frac{1-2^{-4s-1}}{1-2^{-4s-2}}\Big).
\end{align*}
By taking the limit as $s \to 0$, we conclude that the constant Fourier coefficient of $F_{3/2,0}(z)$ equals $8i -8i/(\pi y^{1/2})$. It follows that the constant Fourier coefficient of $\mathcal{H}\vert_{\sigma_0}(z)$ equals
\[
	\frac{-i}{192} \Big(8i - \frac{8i}{\pi \sqrt{y}}\Big) + \frac{1-i}{48} \Big(\frac{-2-2i}{\pi y^{1/2}}\Big)
	= \frac{1}{24} - \frac{1}{8\pi \sqrt{y}},
\]
hence $c_0^+(0) = \frac{1}{24}$ and $c_0^-(0) = - \frac{1}{8\pi}$. 
\end{proof}


\section{Inner Products Involving Regularization Terms} \label{sec:regularization-inner-products}

As before, we fix $k \in \frac{1}{2} \mathbb{Z}$ and $f(z) \in H_k^\sharp(\Gamma_0(N))$ and define $F(z) = y^k \vert f(z) \vert^2$. In section~\S{\ref{sec:automorphic-regularization}}, we showed that $F(z)$ differed from an $L^2$ function $\mathcal{V}_f(z)$ by a sum involving Eisenstein series and theta functions, at least when $k \notin \{\frac{1}{2}, 1\}$. In this section, we relate $\langle \mathcal{V}_f, P_\ell(\cdot,\overline{s}) \rangle$ to $\langle F, P_\ell(\cdot,\overline{s}) \rangle$ by accounting for the contribution of these regularization terms.

To compute the inner products of the form $\langle E_\mathfrak{a}(\cdot, w), P_\ell(\cdot, \overline{s}) \rangle$, we recall the Fourier expansion of $E_\mathfrak{a}(z,w)$ from~\eqref{eq:Fourier-Eisenstein}. We unfold the inner product using the Poincar\'e series as in~\eqref{eq:Poincare-unfolding} to produce
\begin{align} \label{eq:Eisenstein-Poincare}
	\langle E_\mathfrak{a}(\cdot, w), P_\ell(\cdot, \overline{s}) \rangle 
		=\frac{2\pi^{w+\frac{1}{2}}}{(4\pi \ell)^{s-\frac{1}{2}}} \ell^{w-\frac{1}{2}} 
			\varphi_{\mathfrak{a} \infty \ell}(w) \frac{\Gamma(s+w-1)\Gamma(s-w)}{\Gamma(s)\Gamma(w)},
	\qquad
\end{align}
provided that $\Re s >  \frac{1}{2} + \vert \Re w - \frac{1}{2} \vert$ to begin. We write $\varphi_{\mathfrak{a} \infty n}(w) = \varphi_{\mathfrak{a} n}(w)$ for brevity and remark that formulas for these coefficients appear in~\cite{DeshouillersIwaniec82}.

The functions $\varphi_{\mathfrak{a} n}(w)$ have meromorphic continuation in $w$. For $n \neq 0$, they are analytic at $w=1$. Thus~\eqref{eq:Eisenstein-Poincare} implies that
\begin{align}
	\langle \widetilde{E}_\mathfrak{a}(\cdot, 1), P_\ell(\cdot, \overline{s}) \rangle 
		=\frac{\pi \varphi_{\mathfrak{a} \ell}(1)}{(4\pi \ell)^{s-1}}
			 \Gamma(s-1).
\end{align}

Lastly, we consider inner products of the form $\langle y^\frac{1}{2} \vert g(z) \vert^2, P_\ell(\cdot, \overline{s}) \rangle$, in which $g \in M_{1/2}(\Gamma_0(N))$. Suppose that $g(z) = \sum b(n) e(nz)$.  Then
\[
	\langle y^\frac{1}{2} \vert g(z) \vert^2, P_\ell(\cdot, \overline{s}) \rangle
	= \frac{\Gamma(s-\frac{1}{2})}{(4\pi)^{s-\frac{1}{2}}}
		\sum_{n \geq 0} \frac{b(n+\ell)\overline{b(n)}}{(n+\ell)^{s-\frac{1}{2}}}.
\]
By~\cite[Theorem A]{SerreStark77}, $M_{1/2}(\Gamma_0(N))$ is spanned by theta functions of the form $\sum \chi_t(n) e(n^2 t z)$, with $t$ square-free and satisfying $4 \,\mathrm{cond}(\chi_t)^2 t \mid N$, where $\mathrm{cond}(\chi)$ denotes the conductor of $\chi$, $\chi_t(n) = (\frac{t}{n})$ if $t \equiv 1 \bmod 4$, and $\chi_t(n) = (\frac{4t}{n})$ otherwise. We note that $\mathrm{cond}(\chi_t) = t$ if $t \equiv 1 \bmod 4$ and $4t$ otherwise.  In particular, $\{b(n)\}$ is supported on integers of the form $n^2 t$, where $n \in \mathbb{Z}$ and $4t^3 \mid N$. Thus
\[
	\langle y^\frac{1}{2} \vert g(z) \vert^2, P_\ell(\cdot, \overline{s}) \rangle
	= \frac{\Gamma(s-\frac{1}{2})}{(4\pi)^{s-\frac{1}{2}}}
		\sum_{\substack{t_1^3, t_2^3 \mid \frac{N}{4} \\ t_i \text{ square-free}}}
		\sum_{n_1^2 t_1 = n_2^2 t_2 + \ell} 
				\frac{b(n_1^2t_1)\overline{b(n_2^2 t_2)}}{(n_1^2t_1)^{s-\frac{1}{2}}}.
\]

If $M_{1/2}(\Gamma_0(N))$ is one-dimensional (for example, if $\frac{N}{4}$ is cube-free), then $t_1 = t_2=1$. In this case, the inner sum $n_1^2 = n_2^2 + \ell$ has finitely many solutions. Otherwise, the sum may be infinite (depending on $\ell$). Since the solution set $(n_1,n_2)$ of $n_1^2 t_1 = n_2^2 t_2 + \ell$ is exponentially sparse in any case, the sum above always converges for $\Re s > \frac{1}{2}$. Thus $\langle y^\frac{1}{2} \vert g(z) \vert^2, P_\ell(\cdot, \overline{s}) \rangle$ is analytic in $\Re s > \frac{1}{2}$ no matter the dimension of $M_{1/2}(\Gamma_0(N))$.

\begin{remark} \label{rem:theta-shift-continuation}
In fact, $\langle y^\frac{1}{2} \vert g(z) \vert^2, P_\ell(\cdot, \overline{s}) \rangle$ has a meromorphic continuation to all $s \in \mathbb{C}$. To see this, note that the series above is essentially supported on positive integers $x$ satisfying the generalized Pell equation $t_1 x^2 - t_2 y^2 = \ell$. When solutions exist, they lie in finitely many classes of linear recurrences. Splitting $\langle y^\frac{1}{2} \vert g(z) \vert^2, P_\ell(\cdot, \overline{s}) \rangle$ along this subdivision, and splitting further to ignore the effect of the characters $\chi_{t_1}$ and $\chi_{t_2}$, it suffices to continue series of the form $\sum_{m \geq 1} A_m^{-s}$, where $\{A_m\}$ satisfies a degree two linear recurrence. Fortunately, such results are known; see for example~\cite{HolgadoVicente23}, which treats a much more general case. \hfill //
\end{remark}

At this point, it is straightforward to relate the inner products $\langle F, P_\ell(\cdot , \overline{s}) \rangle$ and $\langle \mathcal{V}_f, P_\ell(\cdot, \overline{s}) \rangle$. We record our results in the following proposition.

\begin{proposition} \label{prop:F-regularization}
Let $f \in H_k^\sharp(\Gamma_0(N))$ and set $F(z) = y^k \vert f(z) \vert^2$. For $k=\frac{3}{2}$, 
\begin{align} \label{eq:regularization-inner-products-3/2}
	\langle F, P_\ell(\cdot, \overline{s}) \rangle 
	& = \langle \mathcal{V}_f, P_\ell(\cdot, \overline{s}) \rangle 
	+
		\frac{\sqrt{\pi}\Gamma(s+\frac{1}{2})\Gamma(s-\frac{3}{2})}
					{(4 \pi \ell)^{s-\frac{3}{2}} \Gamma(s)}
				\sum_{\mathfrak{a} : \varkappa_\mathfrak{a} = 0}
						\vert c_\mathfrak{a}^+(0) \vert^2 \varphi_{\mathfrak{a} \ell}(\tfrac{3}{2}) \\
	& \qquad +
	 \frac{2\pi \Gamma(s-1)}{(4\pi \ell)^{s-1}} \sum_{\mathfrak{a} : \varkappa_\mathfrak{a} = 0}
			\Re\big(c_\mathfrak{a}^+(0) \overline{c_\mathfrak{a}^-(0)}\big) \varphi_{\mathfrak{a} \ell}(1) \\
	& \qquad +
	\frac{4\Gamma(s-\frac{1}{2})}{(4\pi)^{s-\frac{1}{2}}} \!\!\!
		\sum_{\substack{t_1^3, t_2^3 \mid \frac{N}{4} \\ t_i \text{ square-free}}}
		\sum_{n_1^2 t_1 = n_2^2 t_2 + \ell} \!\!\!\!
			\frac{a_{\xi f}(n_1^2t_1)\overline{a_{\xi f}(n_2^2 t_2)}}{(n_1^2t_1)^{s-\frac{1}{2}}},
\end{align}
in which $\mathcal{V}_f$ is defined as in~\eqref{eq:regularization-3/2} and $a_{\xi f}(n)$ denotes the $n$-th Fourier coefficient of $\xi_{3/2} f(z)$. For $k < \frac{1}{2}$ in $\frac{1}{2}\mathbb{Z}$, we have instead

\begin{align} \label{eq:regularization-inner-products-general}
	\langle F, P_\ell(\cdot, \overline{s}) \rangle 
	& = \langle \mathcal{V}_f, P_\ell(\cdot, \overline{s}) \rangle  \\
	 & \qquad +
				\frac{2 \pi\Gamma(s+1-k)\Gamma(s+k-2)}
							{(4 \pi \ell)^{s-\frac{1}{2}} (\pi \ell)^{k-\frac{3}{2}}\Gamma(s)\Gamma(2-k)}
				\sum_{\mathfrak{a} : \varkappa_\mathfrak{a} = 0}
						\vert c_\mathfrak{a}^-(0) \vert^2 \varphi_{\mathfrak{a} \ell}(2-k) \\
	& \qquad +
	 \frac{2\pi \Gamma(s-1)}{(4\pi \ell)^{s-1}} \sum_{\mathfrak{a} : \varkappa_\mathfrak{a} = 0}
			\Re\big(c_\mathfrak{a}^+(0) \overline{c_\mathfrak{a}^-(0)}\big) \varphi_{\mathfrak{a} \ell}(1),
\end{align}
in which $\mathcal{V}_f$ is defined as in~\eqref{eq:regularization}. 
\end{proposition}

As a corollary, we specify the contribution of correction terms in the regularization $\mathcal{V}_\mathcal{H}(z)$ of $\mathcal{H}(z)$.

\begin{corollary} \label{cor:H-non-spectral}
Let $\ell_o$ denote the odd-part of $\ell$. In $\Re s > \frac{3}{2}$, we have
\begin{align} \label{eq:H-non-spectral}
	& \langle y^{\frac{3}{2}} \vert \mathcal{H} \vert^2, P_\ell(\cdot, \overline{s}) \rangle  \\
	& \,\, = \langle \mathcal{V}_\mathcal{H}, P_\ell(\cdot, \overline{s}) \rangle 
		+ \frac{\sqrt{\pi}\Gamma(s+\frac{1}{2})\Gamma(s-\frac{3}{2})}
					{(4 \pi \ell)^{s-\frac{3}{2}} \Gamma(s)}
				\Big(\frac{2\sigma_{-2}(\frac{\ell}{4}) - \sigma_{-2}(\frac{\ell}{2})+\sigma_{-2}(\ell_o)}
									{4032\zeta(3)} \Big)  \\
	& \,\,\, - \frac{\Gamma(s-1)}{(4\pi \ell)^{s-1}}
				\Big(\frac{2\sigma_{-1}(\frac{\ell}{4}) - \sigma_{-1}(\frac{\ell}{2})
							+\sigma_{-1}(\ell_o)}{288\zeta(2)} \Big) 
	+ \frac{\Gamma(s-\frac{1}{2})}{32 \pi^{s+\frac{3}{2}}}
		 \hspace{-4.5mm} \sum_{\substack{d \mid \ell \\ d \equiv \frac{\ell}{d} \bmod 2}} \hspace{-3.5mm}
			(d+\tfrac{\ell}{d})^{1-2s}.
\end{align}
\end{corollary}

\begin{proof}
Since $\xi_{3/2} \mathcal{H}(z) = - \frac{1}{16\pi} \theta(z)$ by~\eqref{eq:shadow-definition} and the Fourier expansion~\eqref{eq:H-definition}, we have $a_{\xi \mathcal{H}}(n) = - \frac{1}{16\pi} r_1(n)$. Propositions~\ref{prop:H-cusp-behavior} and~\ref{prop:F-regularization} then give
\begin{align} \label{eq:H-non-spectral-unsimplified}
	& \langle y^{\frac{3}{2}} \vert \mathcal{H} \vert^2, P_\ell(\cdot, \overline{s}) \rangle  \\
	&\quad = \langle \mathcal{V}_\mathcal{H}, P_\ell(\cdot, \overline{s}) \rangle  
	+ 
			\frac{\sqrt{\pi}\, \Gamma(s+\frac{1}{2})\Gamma(s-\frac{3}{2})}
					{(4 \pi \ell)^{s-\frac{3}{2}} \Gamma(s)}
				\Big( \tfrac{1}{144} \varphi_{\infty \ell}(\tfrac{3}{2}) 
					+ \tfrac{1}{576} \varphi_{0 \ell}(\tfrac{3}{2}) \Big) \\
	& \qquad -
	 \frac{\Gamma(s-1)}{(4\pi \ell)^{s-1}}
				\Big(\tfrac{1}{48} \varphi_{\infty \ell}(1) 
					+ \tfrac{1}{96} \varphi_{0 \ell}(1) \Big) 
		+
	\frac{\Gamma(s-\frac{1}{2})}{4(4\pi)^{s+\frac{3}{2}}}
		\sum_{n \geq 0}	\frac{r_1(n+\ell)r_1(n)}{(n+\ell)^{s-\frac{1}{2}}}.
\end{align}

To simplify further, we give explicit descriptions of the Fourier coefficients $\varphi_{\mathfrak{a}\ell}(w)$. Conveniently, the formulas we require appear in~\cite[\S{3.3}]{HKLDWSphere}:
\begin{align}
 \label{eq:phi-formulas}
	\varphi_{0 \ell}(w) &= \frac{\sigma_{1-2w}^{(2)}(\ell)}{4^w \zeta^{(2)}(2w)},
	\,\,\,
	\varphi_{\infty \ell}(w) 
		= \frac{2^{2-4w} \sigma_{1-2w}(\frac{\ell}{4}) - 2^{1-4w} \sigma_{1-2w}(\frac{\ell}{2})}
				{\zeta^{(2)}(2w)},
				\qquad
\end{align}
in which $\zeta^{(2)}(s) = (1-2^{-s})\zeta(s)$, $\sigma_\nu^{(2)}$ denotes the sum-of-divisors function with its $2$-factor removed, and we adopt the convention that $\sigma_\nu(m) = 0$ for $m \notin \mathbb{Z}$. By Euler products, $\sigma_\nu^{(2)}(\ell) = \sigma_\nu(\ell_o)$, where $\ell_0$ is the odd part of $\ell$.

Finally, we note that the series in~\eqref{eq:H-non-spectral-unsimplified} may be written as a divisor sum:
\begin{align*}
	& \sum_{n \geq 0}	\frac{r_1(n+\ell)r_1(n)}{(n+\ell)^{s-\frac{1}{2}}}
		= \sum_{\substack{n_1,n_2 \in \mathbb{Z} \\ n_2^2 - n_1^2 = \ell}} \vert n_2 \vert^{1-2s} \\
	&	\qquad 
	= 2^{2s-1} \!\!\! \sum_{\substack{d_1,d_2 \in \mathbb{Z} \\ d_1 d_2 = \ell \\ d_1 \equiv d_2 \bmod 2}} \!\!\!
		\vert d_1 + d_2 \vert^{1-2s}
	= 2^{2s} \!\!\! \sum_{\substack{d \mid \ell \\ d \equiv \frac{\ell}{d} \bmod 2}} \!\!\!
			(d+\tfrac{\ell}{d})^{1-2s},
\end{align*}
which completes the proof.
\end{proof}


\section{Spectral Expansion and Rightmost Poles} \label{sec:spectral-expansion}

As before, fix $f \in H_k^\sharp(\Gamma_0(N))$ with $k \notin \{\frac{1}{2}, 1\}$ and define $F(z) = y^k \vert f(z) \vert^2$. In this section, we show that the inner product $\langle F, P_\ell(\cdot, \overline{s}) \rangle$ admits meromorphic continuation to $s \in \mathbb{C}$. By Proposition~\ref{prop:F-regularization}, it suffices to consider the inner products $\langle \mathcal{V}_f, P_\ell(\cdot , \overline{s}) \rangle$ instead, as the regularization terms contribute explicit terms which are meromorphic in $\mathbb{C}$, either by inspection or as a consequence of Remark~\ref{rem:theta-shift-continuation}.

Selberg's spectral theorem (cf.~\cite[Theorem 15.5]{IwaniecKowalski04}) gives the following spectral expansion of $P_\ell(z,s)$:
\begin{align} \label{eq:spectral-expansion-template}
	P_\ell(z,s) &= \sum_j \langle P_\ell(\cdot ,s ), \mu_j \rangle \mu_j(z) \\
	& \qquad + \sum_{\mathfrak{a}} \frac{V_N}{4\pi}
		\int_{-\infty}^\infty \langle P_\ell(\cdot, s) ,E_\mathfrak{a}(\cdot, \tfrac{1}{2}+it) \rangle 
			E_\mathfrak{a}(z, \tfrac{1}{2}+it) dt,
\end{align}
in which $\mathfrak{a}$ varies through the cusps of $\Gamma_0(N)$, $V_N = \frac{\pi}{3} \cdot N \prod_{p \mid N}(1+1/p)$ denotes the volume of $\Gamma_0(N) \backslash \mathfrak{h}$, and $\{\mu_j\}$ is an orthonormal Hecke-eigenbasis for the space of weight $0$ Maass cusp forms on $\Gamma_0(N)$. These Maass forms have Fourier expansions at all cusps, which we write in the form
\begin{align} \label{eq:Maass-form-Fourier-expansion}
	\mu_{j\mathfrak{a}}(z) 
	:= \mu_j \vert_{\sigma_{\mathfrak{a}}}(z)
	= y^{\frac{1}{2}} \sum_{ n \neq 0} \rho_{j\mathfrak{a}}(n) K_{it_j}(2\pi \vert n \vert y) e(nx).
\end{align}

Before proceeding, we record two useful lemmas regarding the growth of the coefficients $\rho_{j\mathfrak{a}}(m)$ on average. The first of these concerns the average growth of $\rho_{j\mathfrak{a}}(m)$ with respect to $m$ and is taken from~\cite{Iwaniec02}.

\begin{lemma}[{\cite[(8.7)]{Iwaniec02}}] \label{lem:Maass-form-coefficient-average}
Let $\mu_j$ be an $L^2$-normalized Maass cusp form on $\Gamma_0(N)$ with Fourier expansion at $\mathfrak{a}$ of the form~\eqref{eq:Maass-form-Fourier-expansion}. Then
\[
	\sum_{m \leq M} \vert \rho_{j \mathfrak{a}}(m) \vert^2
	\ll_{N} (M + \vert t_j \vert) e^{\pi \vert t_j \vert}.
\]
\end{lemma}

Our second lemma is a spectral average generalizing~\cite[Theorem 6]{Kuznetsov80}.

\begin{lemma} \label{lem:Kuznetsov-long} 
Let $\{\mu_j\}$ denote an orthonormal basis of Maass cusp forms for $\Gamma_0(N)$, with Fourier expansions given by~\eqref{eq:Maass-form-Fourier-expansion}. For $\ell > 0$ and any $\epsilon > 0$,
\[
	\sum_{\vert t_j \vert \leq X} 
	\frac{\vert \rho_j(\ell) \vert^2}{\cosh \pi t_j }
	= \frac{X^2}{\pi^2} + O_{N,\epsilon}(X \log X + X \ell^\epsilon + \ell^{\frac{1}{2}+\epsilon}).
\]
\end{lemma}

\begin{proof} For level $N=1$, this result is~\cite[Theorem 6]{Kuznetsov80}. More generally, we adapt~\cite[\S{6}]{Kuznetsov80}, replacing the level $1$ trace formula with one on $\Gamma_0(N)$, as found in~\cite[Lemma 4.7]{DeshouillersIwaniec82}, e.g. To carry out this generalization, we require the Kloosterman sum estimate $S_{\infty\infty}(\ell,\ell,c) \ll (\ell,c)^{1/2} c^{1/2} d(c)$ from~\cite[Lemma 2.6]{DeshouillersIwaniec82} as well as the Eisenstein series coefficient estimate
\begin{align} \label{eq:Eisenstein-Fourier-coefficient-bound}
	\varphi_{\mathfrak{a} \ell}(\tfrac{1}{2}+it)
	\ll_{N} \frac{d(\ell)}{\vert \zeta(1+2it) \vert}
	\ll_{N,\epsilon} d(\ell) \log t,
\end{align}
To see~\eqref{eq:Eisenstein-Fourier-coefficient-bound}, one may represent $E_\mathfrak{a}(z,s)$ in terms of Eisenstein series attached to characters via~\cite[Theorem 6.1]{Young19} and apply the Fourier coefficient formulas in~\cite[Proposition 4.1]{Young19}, then apply~\cite[(3.11.10)]{Titchmarsh86}.
\end{proof}

Continuing on, substitution of~\eqref{eq:spectral-expansion-template} into $\langle \mathcal{V}_f, P_h(\cdot, \overline{s}) \rangle$ produces 
\begin{align} \label{eq:spectral-expansion-unsimplified}
	\langle \mathcal{V}_f, P_\ell(\cdot , \overline{s}) \rangle
	 &= \sum_j \langle \mu_j, P_\ell(\cdot , \overline{s})\rangle \langle \mathcal{V}_f, \mu_j(z) \rangle \\
	& \quad 
		+ \frac{V_N}{4\pi} \sum_\mathfrak{a}
		\int_{-\infty}^\infty
			\langle E_\mathfrak{a}(\cdot, \tfrac{1}{2}+it),  P_\ell(\cdot, \overline{s})\rangle
			\langle \mathcal{V}_f, E_\mathfrak{a}(z, \tfrac{1}{2}+it) \rangle dt,
\end{align}
which we call the \emph{spectral expansion of $\langle \mathcal{V}_f, P_\ell(\cdot , \overline{s}) \rangle$}. We will refer to the terms at right in~\eqref{eq:spectral-expansion-unsimplified} as the \emph{discrete spectrum} and \emph{continuous spectrum}, respectively. To make this more explicit, we apply~\eqref{eq:Eisenstein-Poincare} and the formula
\[
	\langle \mu_j , P_\ell(\cdot, \overline{s}) \rangle
		= \frac{\rho_j(\ell) \sqrt{\pi}}{(4\pi \ell)^{s-\frac{1}{2}}}
			\frac{\Gamma(s-\frac{1}{2}-it_j) \Gamma(s-\frac{1}{2}+it_j)}{\Gamma(s)},
\]
which follows from~\cite[6.621(3)]{GradshteynRyzhik07}. We conclude that $\langle \mathcal{V}_f, P_\ell(\cdot, \overline{s}) \rangle$ admits a spectral decomposition of the form $\Sigma_{\mathrm{disc}}(s) + \Sigma_{\mathrm{cont}}(s)$, in which
\begin{align*}
	\Sigma_{\mathrm{disc}}(s) 
		& :=  \frac{\sqrt{\pi}}{(4\pi \ell)^{s-\frac{1}{2}} \Gamma(s)}
		\sum_j \rho_j(\ell) \Gamma(s - \tfrac{1}{2} + it_j)\Gamma(s- \tfrac{1}{2} -it_j)
			\langle \mathcal{V}_f, \mu_j \rangle; \\
	\Sigma_{\mathrm{cont}}(s)
		& := \frac{V_N}{2} \sum_\mathfrak{a}
		\int_{-\infty}^\infty 
			\frac{\varphi_{\mathfrak{a} \ell}(\frac{1}{2}+it) \Gamma(s-\frac{1}{2}+it)\Gamma(s-\frac{1}{2}-it)}
				{(4\pi \ell)^{s-\frac{1}{2}} (\pi \ell)^{-it} \Gamma(s)\Gamma(\frac{1}{2}+it)} \\
		& \hspace{65 mm} \times
			\langle \mathcal{V}_f, E_\mathfrak{a}(\cdot, \tfrac{1}{2}+it) \rangle dt.
\end{align*}

This spectral expansion is initially defined for $\Re s > 1+ \vert k - 1 \vert$, provided all expressions converge. Fortunately, convergence is not an issue: 

\begin{lemma} The functions $\Sigma_{\mathrm{disc}}(s)$ and $\Sigma_{\mathrm{cont}}(s)$ converge for all $s \in \mathbb{C}$ away from their poles.
\end{lemma}

\begin{proof}
In the discrete spectrum $\Sigma_{\mathrm{disc}}(s)$, this follows from Lemma~\ref{lem:Kuznetsov-long}, Stirling's approximation (providing decay of size $e^{-\pi \vert t_j \vert}$ for fixed $s$), and the trivial estimate $\vert \langle \mathcal{V}_f, \mu_j \rangle  \vert \ll \Vert \mathcal{V}_f \Vert^{1/2} \cdot \Vert \mu_j \Vert^{1/2} \ll_f 1$. (This estimate is \emph{very} weak, and will be improved in section~\S{\ref{sec:jutila-triple-inner-products}}.)

In the continuous spectrum, convergence follows from Stirling's approximation, weak upper bounds for $\langle \mathcal{V}_f, E_\mathfrak{a}(\cdot, \frac{1}{2}+it) \rangle$ derived from the Rankin--Selberg method and Phragm\'{e}n--Lindel\"{o}f convexity principle, and~\eqref{eq:Eisenstein-Fourier-coefficient-bound}.
\end{proof}

Thus $\Sigma_{\mathrm{disc}}(s)$ defines a meromorphic function on the entire complex plane, with potential poles at $s= \frac{1}{2} \pm it_j - m$ for integer $m \geq 0$ and any spectral type $t_j$. It is analytic in the right half-plane $\Re s > \frac{1}{2} + \Theta$, where $\Theta \leq 7/64$ denotes partial progress towards the Ramanujan--Petersson conjecture~\cite{KimSarnak03}.

The continuous spectrum $\Sigma_{\mathrm{cont}}(s)$ also has meromorphic continuation to all $s$, though the precise continuation to $\Re s > - M$ involves both $\Sigma_{\mathrm{cont}}(s)$ and $O(M)$ residue terms extracted through contour shifts of the integral in $\Sigma_{\mathrm{cont}}$. For a discussion of the continuation process in a similar case, we refer the reader to~\cite[\S{4}]{HoffsteinHulse13} or~\cite[\S{3.3.2}]{HKLDWSphere}. The continuous spectrum is clearly analytic in $\Re s > \frac{1}{2}$.

Thus $\langle \mathcal{V}_f, P_\ell(\cdot, \overline{s}) \rangle$, originally defined for $\Re s > 1+\vert k -1 \vert$, extends meromorphically to a function on the entire complex plane. Since it is analytic in $\Re s > \frac{1}{2}+ \Theta$, any pole of $\langle F, P_\ell(\cdot, \overline{s}) \rangle$ in $\Re s > \frac{1}{2} + \Theta$ occurs as a pole of the explicit regularization factors presented in Proposition~\ref{prop:F-regularization}.

\begin{remark} \label{rem:H-specialization}
In section~\S{\ref{sec:unfolding}}, we gave the general decomposition
\[
	\langle F, P_\ell(\cdot, \overline{s}) \rangle 
	= I_\ell^+(s) + I_\ell^0(s) + I_\ell^\times(s) + I_\ell^-(s).
\]
The two terms $I_\ell^0(s)$ and $I_\ell^\times(s)$ are finite sums and inherit meromorphic continuation to $s \in \mathbb{C}$ from the continuations of $G_k$ and the ${}_2F_1$-hypergeometric function. Thus the continuation of $\langle F, P_\ell(\cdot, \overline{s}) \rangle$ implies a continuation for $I_\ell^+(s) + I_\ell^-(s)$. It is possible, albeit challenging, to establish the meromorphic continuations of $I_\ell^+(s)$ and $I_\ell^-(s)$ as separate entities. Here, the idea is to first continue $I_\ell^-(s)$ by relating it to the Dirichlet series
\[
	\sum_{n =1}^\infty \frac{a_{\xi f}(n) \overline{a_{\xi f}(n+\ell)}}{(n+\ell)^{s-w-k} n^{w+1}},
\]
which admits meromorphic continuation through relation to the triple inner product $\langle y^{2-k} \vert \xi_k f \vert^2, P_\ell(\cdot ,\overline{s}) \rangle$. Establishing this continuation is not so difficult when $\Re w < -1$, but in practice we require $\Re w$ as large as $-k$ (to evaluate a particular contour integral representation), and this creates major complications in weights $k<1$.

Fortunately, these problems disappear altogether for $f= \mathcal{H}$, as the series $I_\ell^-(s)$ defines a \emph{finite} sum in this case (cf.~\S{\ref{sec:H-decomposition}}). To simplify the exposition in this work, we narrow our typical focus from $f \in H_k^\sharp(\Gamma_0(N))$ to $f = \mathcal{H}$. Still, the construction of the meromorphic continuations of $I_\ell^\pm(s)$ and estimates for shifted convolutions of generic mock modular forms of polynomial growth are of independent interest and will appear in future work. \hfill //
\end{remark}

\subsection{Classifying the Rightmost Poles of \texorpdfstring{$D_\ell(s)$}{Dh(s)}} \label{ssec:D_h-poles}

As an application of our work thus far, we classify the rightmost poles of the shifted convolution Dirichlet series $D_\ell(s)$ from~\eqref{eq:D_h-definition-redux}. We prove the following theorem.

\begin{theorem} \label{thm:D_h-poles}
The Dirichlet series $D_\ell(s)$ is analytic in the right half-plane $\Re s > \frac{3}{2}$ and extends meromorphically to all $s \in \mathbb{C}$. If $\ell \equiv 2 \bmod 4$, then $D_\ell(s) = 0$ identically. Otherwise, $D_\ell(s)$ has two simple poles in the right half-plane $\Re s > \frac{1}{2}$, at $s=\frac{3}{2}$ and $s=1$, with residues
\begin{align*}
	\Res_{s=\frac{3}{2}} D_\ell(s) &= 
			 \frac{\pi^2}{126\zeta(3)}
				\big(2\sigma_{-2}(\tfrac{\ell}{4}) - \sigma_{-2}(\tfrac{\ell}{2})+\sigma_{-2}(\ell_o) \big); \\
	\Res_{s=1} D_\ell(s) &=
				-\frac{1}{3\pi}
				\big(2\sigma_{-1}(\tfrac{\ell}{4}) - \sigma_{-1}(\tfrac{\ell}{2})+\sigma_{-1}(\ell_o) \big).
\end{align*}
The function $D_\ell(s)$ is otherwise analytic in $\Re s > \frac{1}{2}$.
\end{theorem}

\begin{proof}
Equation~\eqref{eq:Hurwitz-inner-product-decomposition} relates $D_\ell(s)$ to $\langle y^{3/2} \vert \mathcal{H} \vert^2, P_\ell(\cdot, \overline{s}) \rangle$ and Corollary~\ref{cor:H-non-spectral} relates $\langle y^{3/2} \vert \mathcal{H} \vert^2, P_\ell(\cdot, \overline{s}) \rangle$ to $\langle \mathcal{V}_\mathcal{H}, P_\ell(\cdot, \overline{s}) \rangle$. When combined, we produce
\begin{align} \label{eq:D_h-decomposition-full}
	D_\ell(s) &=
		\frac{\pi^{\frac{5}{2}} \Gamma(s-\frac{3}{2})}
					{\ell^{s-\frac{3}{2}} \Gamma(s)}
				\Big(\frac{2\sigma_{-2}(\frac{\ell}{4}) - \sigma_{-2}(\frac{\ell}{2})+\sigma_{-2}(\ell_o)}
									{252\zeta(3)} \Big) \\
		& \quad
		 - \frac{\pi^{\frac{3}{2}} \Gamma(s-1)}{\ell^{s-1} \Gamma(s+\frac{1}{2})}
				\Big(\frac{2\sigma_{-1}(\frac{\ell}{4}) - \sigma_{-1}(\frac{\ell}{2})+\sigma_{-1}(\ell_o)}
				{36\zeta(2)} \Big) \\
		& \quad
		+	\frac{(4\pi)^{s+\frac{1}{2}}
				\langle \mathcal{V}_\mathcal{H}, P_\ell(\cdot, \overline{s}) \rangle}
				{\Gamma(s+\frac{1}{2})}
			+ \frac{ H(\ell)}{12 \ell^{s+\frac{1}{2}} (s-\frac{1}{2})} 
	   + \frac{H(\ell)\Gamma(s)}{4 \sqrt{\pi} \ell^s \Gamma(s+\frac{1}{2})} \\
	& \quad
	+ \frac{r_1(\ell) }{32 \pi \ell^{s-\frac{1}{2}} s (s-\frac{1}{2})}
		+ \frac{r_1(\ell) \Gamma(s)}{96 \sqrt{\pi} \ell^s \Gamma(s+\frac{3}{2})} \\
	& \quad
		- \frac{\Gamma(s)}{\Gamma(s+\frac{3}{2})} 
		\sum_{m=1}^{\ell-1} \frac{H(\ell-m)  r_1(m)}{8 \sqrt{\pi} m^s} 
		{}_2F_1\Big(\begin{matrix} s, s+\frac{1}{2} \\ s+\frac{3}{2} \end{matrix}
					\Big\vert 1-\frac{\ell}{m} \Big) \\
	& \quad 
	- \sum_{\substack{m_1^2 - m_2^2 =\ell \\ m_1,m_2 \geq 1}}
		\frac{m_1 m_2  G_{3/2}(s,m_2^2, m_1^2) }{16\pi \Gamma(s+\frac{1}{2})}
	+ \frac{2^{2s-4}}{\pi (s-\frac{1}{2})}
		 \!\!\!\! \sum_{\substack{d \mid \ell \\ d \equiv \frac{\ell}{d} \bmod 2}} \!\!\!
			(d+\tfrac{\ell}{d})^{1-2s}.
\end{align}
Recall that $\langle \mathcal{V}_\mathcal{H}, P_\ell(\cdot, \overline{s}) \rangle$ is analytic in $\Re s > \frac{1}{2} + \Theta$. By Huxley's resolution of the Selberg eigenvalue conjecture in low level~\cite{Huxley85}, the inner product is in fact analytic in $\Re s > \frac{1}{2}$. Thus, by previous comments, all but the first two terms at right above are analytic in $\Re s > \frac{1}{2}$. Computation of residues completes the proof.
\end{proof}

Since $D_\ell(s)$ has non-negative coefficients, the Wiener--Ikehara theorem (see, e.g.~\cite[Corollary 8.8]{MontgomeryVaughn06}) immediately produces the following:

\begin{corollary} \label{cor:Wiener-Ikehara}
For fixed $\ell$, as $X \to \infty$ we have
\[
	\sum_{n \leq X} H(n)H(n+\ell)
		\sim \frac{\pi^2 X^2}{252\zeta(3)}
				\big(2\sigma_{-2}(\tfrac{\ell}{4}) - \sigma_{-2}(\tfrac{\ell}{2})+\sigma_{-2}(\ell_o) \big).
\]
\end{corollary}


\section{Bounding \texorpdfstring{$D_\ell(s)$}{Dh(s)} in Vertical Strips} \label{sec:D_h-growth-bounds}

To quantify the rate of convergence in Corollary~\ref{cor:Wiener-Ikehara}, we require additional information about the meromorphic properties of $D_\ell(s)$. Specifically, we require uniform estimates for the growth of $D_\ell(s)$ with respect to $\vert \Im s \vert$ in vertical strips outside the domain of absolute convergence.

It suffices to produce growth estimates for each component of the decomposition of $D_\ell(s)$ given in~\eqref{eq:D_h-decomposition-full}. In this section, we produce uniform estimates for every term besides $(4\pi)^{s+\frac{1}{2}} \langle \mathcal{V}_\mathcal{H}, P_\ell(\cdot, \overline{s}) \rangle / \Gamma(s+\frac{1}{2})$, which requires more involved techniques. We prove the following proposition.

\begin{proposition} \label{prop:D_h-preliminary-growth}
Fix $s$ with $\Re s > 0$. Away from poles of $D_\ell(s)$, we have
\[
	D_\ell(s) \ll_\epsilon
		\ell^{\frac{1}{2}-\Re s + \epsilon} \vert s \vert^{-\frac{3}{2}}
		\big(\ell
		+ \vert s \vert \big)
		+ \Big\vert
			\frac{\langle \mathcal{V}_\mathcal{H}, P_\ell(\cdot, \overline{s}) \rangle}{\Gamma(s+\frac{1}{2})}
			\Big\vert
\]
for all $\epsilon > 0$.
\end{proposition}

Our proof of Proposition~\ref{prop:D_h-preliminary-growth} requires a few lemmas. The first of these is a simple upper bound for the Hurwitz class number.

\begin{lemma} \label{lem:H-growth}
We have $H(\ell) \ll_\epsilon \ell^{\frac{1}{2} + \epsilon}$ for all $\epsilon > 0$.
\end{lemma}

\begin{proof} The moment estimate~\eqref{eq:class-number-moment} implies $\widetilde{h}(-\ell) \ll_\epsilon \ell^{\frac{1}{2} + \epsilon}$ for all $\epsilon > 0$. Since $h(-\ell) = \sum_{d^2 \mid \ell} \widetilde{h}(-\ell/d^2)$, we have $h(-\ell) \ll \sum_{d \geq 1} (\ell/d^2)^{\frac{1}{2}+\epsilon} \ll \ell^{\frac{1}{2}+\epsilon}$. The same bound holds for $H(\ell)$, since $H(\ell ) = h(-\ell) + O(1)$.
\end{proof}

We also require uniform estimates for the ${}_2F_1$-hypergeometric function and the function $G_{3/2}$, which are provided by the following two lemmas. 

\begin{lemma} \label{lem:hypergeometric-growth}
For $1 \leq m \leq \ell-1$ and $\Re s > 0$, we have
\[
{}_2F_1\Big(\begin{matrix} s, s+\frac{1}{2} \\ s+\frac{3}{2} \end{matrix}\Big\vert 1-\frac{\ell}{m} \Big)
	\ll \Big(\frac{m}{\ell}\Big)^{\Re s}.
\]
\end{lemma}

\begin{proof}
Following~\cite[9.131(1)]{GradshteynRyzhik07} and the Euler integral~\cite[9.111]{GradshteynRyzhik07},
\begin{align} \label{eq:2F1-transform-and-integral}
	{}_2F_1\Big(\begin{matrix} s, s+\frac{1}{2} \\ s+\frac{3}{2} \end{matrix}
					\Big\vert 1-\frac{\ell}{m} \Big)
		&= \Big(\frac{\ell}{m}\Big)^{1-s}
		{}_2F_1\Big(\begin{matrix} \frac{3}{2}, 1 \\ s+\frac{3}{2} \end{matrix}
					\Big\vert 1-\frac{\ell}{m} \Big) \\
	&=  \Big(\frac{\ell}{m}\Big)^{1-s} (s+\tfrac{1}{2})
		\int_0^1 \frac{(1-t)^{s-\frac{1}{2}} \, dt}{(1-(1-\frac{\ell}{m}) t)^{3/2}}
\end{align}
in the region $\Re s > -\frac{1}{2}$. In this form, we recognize that the hypergeometric function at right in~\eqref{eq:2F1-transform-and-integral} is bounded by ${}_2F_1(\frac{3}{2},1,\frac{3}{2} \vert 1 - \frac{\ell}{m})$ when $\Re s > 0$. To conclude, note that ${}_2F_1(\frac{3}{2},1,\frac{3}{2} \vert 1 - \frac{\ell}{m}) = \frac{m}{\ell}$ by~\cite[(9.121)]{GradshteynRyzhik07}.
\end{proof}

\begin{lemma} \label{lem:G_k-growth}
Fix $\epsilon > 0$. In the region $\Re s > 0$, the function $G_{3/2}(s,n,n+\ell)$ defined in~\eqref{eq:I-minus-definition} satisfies
\[
	G_{\frac{3}{2}}(s, n,n+\ell) \ll
	\frac{\vert s \vert^{\Re s - 2 + \epsilon}}{(n+\ell)^{\Re s}}
	e^{-\frac{\pi}{2} \vert \Im s \vert} \Big(\frac{\vert s \vert^{\frac{1}{2}}}{\sqrt{n+\ell}} + \frac{1}{\sqrt{n}} \Big).
\]
\end{lemma}

\begin{proof} We begin with the contour integral representation~\cite[(8.6.12)]{DLMF}
\begin{align} \label{eq:incomplete-gamma-integral-representation}
	\Gamma(1-k, y) e^y =
		- \frac{y^{-k}}{\Gamma(k)} \cdot \frac{\pi}{2\pi i}
		\int_C \frac{\Gamma(w+k) y^{-w}}{\sin(\pi w)} \,dw,
\end{align}
where $C$ is a contour separating the poles of $\Gamma(w+k)$ from those at $w=0,1,\ldots$ arising from $1/\sin(\pi w)$. Here we require $k \notin -\mathbb{N}$. For $k> 0$, we may take $C$ as a vertical line with $\Re w = - \epsilon$. We apply~\eqref{eq:incomplete-gamma-integral-representation} and the Mellin transform~\cite[(8.14.4)]{DLMF} to $G_k(s,n,n+\ell)$ to write
\begin{align} 
	G_k &= -\frac{\pi n^{-k}}{\Gamma(k)} \frac{1}{2\pi i} \int_{(-\epsilon)}
		\frac{\Gamma(w+k)}{\sin(\pi w) n^{w}}
			\Big(\int_0^\infty y^{s-1-w} \Gamma(1-k, (n+\ell)y) \frac{dy}{y}\Big)  dw \\
			\label{eq:G_k-integral-representation}
	&=
		-\frac{\pi}{\Gamma(k)} \frac{1}{2\pi i} \int_{(-\epsilon)}
			\frac{\Gamma(w+k)\Gamma(s-w-k) \csc(\pi w)}{ n^{k+w} (n+\ell)^{s-w-1} (s-w-1)} dw,
\end{align}
provided $\Re s > \max(1,k)$ to begin. Shifting the contour of integration to $\Re w = - \max(1,k) - \epsilon$ passes finitely many poles from $\csc(\pi w)$ and gives a meromorphic continuation of $G_k$ to $\Re s > 0$ when $k>0$.

We now specialize to $k=\frac{3}{2}$. The contour shift in~\eqref{eq:G_k-integral-representation} to $\Re w = -\frac{3}{2} - \epsilon$ passes a single pole at $w=-1$, with residue 
\[
	\frac{2 \Gamma(s-\frac{1}{2})}{\sqrt{n} (n+\ell)^s s} 
	\ll \frac{1}{\sqrt{n} (n+ \ell)^{\Re s }} \vert s \vert^{\Re s -2} e^{-\frac{\pi}{2} \vert \Im s \vert}.
\]
Stirling shows that the integrand decays exponentially in $\vert \Im w \vert$, for any $s$. We may therefore truncate the integral to $\vert \Im w \vert \leq \frac{1}{2} \vert \Im s \vert$. In this range, the estimates $\vert s - w - k \vert \asymp \vert s \vert$ and $e^{-\frac{\pi}{2} \vert \Im (s-w) \vert} \ll e^{\frac{\pi}{2} \vert \Im w \vert - \frac{\pi}{2} \vert \Im s \vert}$ allow us to extract the $s$-dependence of the integrand. Hence the shifted integral~\eqref{eq:G_k-integral-representation} is $O( (n+\ell)^{-\Re s - \frac{1}{2}} \vert s \vert^{\Re s - \frac{3}{2}+\epsilon} e^{-\frac{\pi}{2} \vert \Im s \vert})$, which completes the proof.
\end{proof}

\begin{proof}[Proof of Proposition~\ref{prop:D_h-preliminary-growth}]
Lemma~\ref{lem:H-growth} and the divisor estimates $\sigma_{-2}(\ell) \ll 1$ and $\sigma_{-1}(\ell) \ll \ell^\epsilon$ imply that the terms at right in the first four lines of~\eqref{eq:D_h-decomposition-full} (excluding the term containing $\langle \mathcal{V}_\mathcal{H}, P_\ell(\cdot, \overline{s}) \rangle$) are
\begin{align} \label{eq:simplest-terms-bounds}
	O_{\Re s, \epsilon}
		\big( \ell^{\frac{3}{2}-\Re s} \vert s \vert^{-\frac{3}{2}}
		+ \ell^{1-\Re s + \epsilon} \vert s \vert^{-1}
		+ \ell^{\frac{1}{2} - \Re s + \epsilon} \vert s \vert^{-\frac{1}{2}}
		\big).
\end{align}
By factoring this upper bound in the form $\ell^{\frac{1}{2} - \Re s + \epsilon} \vert s \vert^{-\frac{3}{2}}( \ell + \ell^{\frac{1}{2}} \vert s \vert^{\frac{1}{2}} + \vert s \vert)$, we observe that the second summand is always dominated by the first or third term, and may be ignored.

It remains to estimate the three terms in the last two lines of~\eqref{eq:D_h-decomposition-full}. 
We first consider the divisor sum. In the right half-plane $\Re s > \frac{1}{2}$, we bound $\vert d + \ell /d \vert^{1-2\Re s} \ll \ell^{\frac{1}{2}-\Re s}$, so the divisor sum is $O( \ell^{\frac{1}{2} - \Re s + \epsilon} \vert s \vert^{-1})$, which is non-dominant. Otherwise, if $\Re s < \frac{1}{2}$, we bound $\vert d + \ell/d \vert^{1-2\Re s} \ll \ell^{1-2\Re s}$, so the full divisor sum is $O(\ell^{1-2\Re s + \epsilon} \vert s \vert^{-1})$. This term is dominated by the second term of~\eqref{eq:simplest-terms-bounds} when $\Re s > 0$.

We next consider the contribution of the hypergeometric term in~\eqref{eq:D_h-decomposition-full}. By Lemma~\ref{lem:hypergeometric-growth}, Stirling's formula, and then Lemma~\ref{lem:H-growth}, this term is
\begin{align*}
	&\ll_{\Re s} \ell^{-\Re s} \vert s \vert^{-\frac{3}{2}} \sum_{m=1}^{\ell-1} H(\ell-m) r_1(m)
	\ll_{\Re s,\epsilon} \ell^{1- \Re s + \epsilon} \vert s \vert^{-\frac{3}{2}}
\end{align*}
in the region $\Re s > 0$. Note that this term is dominated by the first error term in~\eqref{eq:simplest-terms-bounds}. 

Finally, we consider the term in~\eqref{eq:D_h-decomposition-full} involving $G_{3/2}(s,m_2^2,m_1^2)$. By Lemma~\ref{lem:G_k-growth}, this term is
\begin{align} \label{eq:G_k-application-bound}
	O_{\Re s} \Big(\vert s \vert^{-2+\epsilon} \sum_{m_1^2 - m_2^2 = \ell} \frac{m_1 m_2}{m_1^{2\Re s}}
	\Big( \frac{\vert s \vert^{\frac{1}{2}}}{m_1} + \frac{1}{m_2} \Big) \Big)
\end{align}
in the region $\Re s > 0$. The contribution of $1/m_2$ in the parenthetical is 
\[
	\ll_{\Re s}
	\vert s \vert^{-2+\epsilon} \sum_{m_1^2 - m_2^2 = \ell} \frac{1}{m_1^{2\Re s -1}}
	\ll_{\Re s} \vert s \vert^{-2+\epsilon} \ell^\epsilon \big(\ell^{\frac{1}{2} - \Re s} + \ell^{1-2\Re s} \big),
\]
in which we've used that $\sqrt{\ell} \leq m_1 \leq \ell$ and that the sum has at most $d(\ell)$ terms. Since $m_1 \geq m_2$, the contribution of the other term in the parenthetical of~\eqref{eq:G_k-application-bound} is at most $\vert s \vert^{1/2}$ times larger. Both upper bounds are majorized by the contribution of the divisor sum in~\eqref{eq:D_h-decomposition-full}.
\end{proof}


\section{Non-Cuspidal Spectral Inner Products} \label{sec:jutila-triple-inner-products}

To bound $\langle \mathcal{V}_\mathcal{H}, P_\ell(\cdot, \overline{s}) \rangle$ in vertical strips, we apply the spectral expansion $\langle \mathcal{V}_\mathcal{H}, P_\ell(\cdot, \overline{s}) \rangle = \Sigma_{\mathrm{disc}}(s) + \Sigma_{\mathrm{cont}}(s)$ computed in section~\S{\ref{sec:spectral-expansion}}. 
In the discrete spectrum, Stirling's approximation, dyadic subdivision, Cauchy--Schwarz, and Lemma~\ref{lem:Kuznetsov-long} reduce our task to bounding the inner products $\langle \mathcal{V}_\mathcal{H}, \mu_j \rangle$.
Since Maass cusp forms are orthogonal to Eisenstein series and to norm-squares of theta functions (cf.~\cite[Remark 2]{Nelson21}), this is equivalent to bounding the unregularized inner products $\langle y^{3/2} \vert \mathcal{H} \vert^2 , \mu_j \rangle$.

While good estimates for inner products of the form $\langle y^k \vert f \vert^2 , \mu_j \rangle$ are known when $f$ is a holomorphic cusp form or Maass cusp form (at least on average), the non-cuspidal nature of $\mathcal{H}$ meters the applicability of prior results. Fortunately, it is possible to modify work of Jutila~\cite{Jutila96,Jutila97} in the Maass cusp form case to address the case of harmonic Maass forms. 
Working in a somewhat general setting, we prove the following theorem.

\begin{theorem} \label{thm:triple-inner-product-bound}
Fix $f \in H_k^\sharp(\Gamma_0(N))$ with $k \in \frac{1}{2} + \mathbb{Z}$. Let $\mu_j(z)$ be an $L^2$-normalized Hecke--Maass cusp form of weight $0$ on $\Gamma_0(N)$, with spectral type $t_j \in \mathbb{R}$. For all $\epsilon > 0$, we have
\[
	\langle y^k \vert f \vert^2, \mu_j \rangle 
		\ll 
		\big( \vert t_j \vert^{2k-1 + \epsilon} + \vert t_j \vert^{3-2k + \epsilon} \big)
		e^{-\frac{\pi}{2} \vert t_j \vert}.
\]
\end{theorem}

Our proof of Theorem~\ref{thm:triple-inner-product-bound} follows the general method of~\cite{Jutila96,Jutila97}. Very roughly, this plan involves two steps:
\begin{enumerate}
	\item[a.] We relate $\langle y^k \vert f \vert^2, \mu_j \rangle$, which is an integral over $\Gamma_0(N) \backslash \mathfrak{h}$, to an `unfolded' integral over $\Gamma_\infty \backslash \mathfrak{h}$, by introducing an Eisenstein series as an unfolding object. This technique was developed in~\cite[\S{2}]{Jutila96} for $f$ a level $1$ holomorphic or Maass cusp form, and we adapt it to the case of $f \in H_k^\sharp(\Gamma_0(N))$.
	\item[b.] The unfolded integral can be understood as an integral transform of a sum involving Fourier coefficients of $f$ and $\mu_j$ at various cusps. We truncate the sums and integrals and apply estimates for the Fourier coefficients of $f$ and $\mu_j$ to bound the truncations.
\end{enumerate}
We remark that~\cite{Jutila97} also applies the spectral large sieve, to produce a fairly sharp upper bound for the spectral average $\sum_{\vert t_j \vert \sim T} \vert \langle y^k \vert f \vert^2, \mu_j \rangle \vert^2 e^{\pi \vert t_j \vert}$. The spectral large sieve seems unlikely to give quantitative improvements in the non-cuspidal case, so we focus on point-wise bounds in this section.

Though not the main focus of this work, we remark that Theorem~\ref{thm:triple-inner-product-bound} has applications to modular forms of half-integral weight, since $M_k(\Gamma_0(N)) \subset H_k^\sharp(\Gamma_0(N))$. For convenient reference, we present this as a corollary.

\begin{corollary} \label{cor:modular-form-specialization}
Fix $k \in \frac{1}{2} + \mathbb{Z}$ and $f \in M_k(\Gamma_0(N))$. Let $\mu_j(z)$ be an $L^2$-normalized Hecke--Maass cusp form of weight $0$ on $\Gamma_0(N)$, with spectral type $t_j \in \mathbb{R}$. For all $\epsilon > 0$, we have
\[
	\langle y^k \vert f \vert^2, \mu_j \rangle 
		\ll \big(\vert t_j \vert^{2k-1 + \epsilon} + \vert t_j \vert^{1+\epsilon}\big)
		e^{-\frac{\pi}{2} \vert t_j \vert}.
\]
\end{corollary}

We remark that Corollary~\ref{cor:modular-form-specialization} improves certain technical results in~\cite{Kiral15}. In particular, we improve the $t_j$-dependence of~\cite[Proposition 14]{Kiral15} in any case that our result applies.

\subsection{Jutila's Extension of the Rankin--Selberg Method}

The material in this section adapts~\cite[\S{2}]{Jutila96} from $\mathrm{SL}(2,\mathbb{Z})$ to $\Gamma_0(N)$. Let $\phi(z)$ be an $L^2$ function on $\Gamma_0(N) \backslash \mathfrak{h}$ and let $E_\infty(z,s)$ denote the weight $0$ Eisenstein series at the cusp $\infty$ of $\Gamma_0(N)$. By multiplying and dividing by $E_\infty(z,s)$ and unfolding, we find
\[
		\iint_{\Gamma_0(N) \backslash \mathfrak{h}} \phi(z) \frac{dxdy}{y^2}
			= \int_0^\infty \int_0^1 \frac{\phi(z) y^s}{E_\infty(z,s)} \frac{dxdy}{y^2},
\]
at first for $\Re s > 1$. Since $E_\infty(z,s)$ has a simple pole at $s=1$ with residue $\frac{3}{\pi} \cdot [\Gamma_0(N) : \mathrm{SL}(2,\mathbb{Z}) ]^{-1} = V_N^{-1}$, we can remove the denominator $E_\infty(z,s)$ by letting $s \to 1^+$ to produce
\begin{align} \label{eq:unnatural-unfolding}
	 \iint_{\Gamma_0(N) \backslash \mathfrak{h}} \phi(z) \frac{dxdy}{y^2}
	&\phantom{:}= V_N \lim_{s \to 1^+} (s-1) R(\phi, s); \\
	R(\phi, s) &:= \int_0^\infty\int_0^1 \phi(z) y^{s-1} \frac{dxdy}{y},
\end{align}
in which $R(\phi, s)$ is the typical Rankin--Selberg transform of $\phi$.

We define $R^*(\phi,s) = \zeta^*(2s) R(\phi,s)$ and $R^*_0(\phi,s) = s(s-1)R^*(\phi,s)$, so that~\eqref{eq:unnatural-unfolding} equals $\frac{\pi}{6} V_N \Res_{s=1} R^*(\phi ,s) = \frac{\pi}{6} V_N R^*_0(\phi,1)$. Note that $R_0^*(\phi,s)$ is entire, in part because $\phi \in L^2$. By the residue theorem,
\[
	R_0^*(\phi,1) = \frac{1}{2\pi i} \int_\mathcal{O} g(s) \frac{R_0^*(\phi, s)}{s-1} ds,
\]
in which $\mathcal{O}$ is a contour encircling $s=1$ once counterclockwise and $g(s)$ is a rapidly decaying holomorphic function satisfying $g(1)=1$. We bend $\mathcal{O}$ into a rectangle connecting $a \pm i T$ and $1-a \pm iT$, then let $T \to \infty$ and use decay in $g$ to render the horizontal components of $\mathcal{O}$ negligible. It follows that
\begin{align} 
\begin{split} \label{eq:Rankin-Selberg-two-lines}
	R_0^*(\phi, 1)
		&= \frac{1}{2\pi i} \int_{(a)} g(s) \frac{R_0^*(\phi, s)}{s-1} ds
			- \frac{1}{2\pi i} \int_{(1-a)} g(s) \frac{R_0^*(\phi, s)}{s-1} ds \\
		&= \frac{1}{2\pi i} \int_{(a)} 
			\Big(\frac{g(s)}{s-1}R_0^*(\phi, s) + \frac{g(1-s)}{s} R_0^*(\phi, 1-s) \Big) ds.
\end{split}
\end{align}

We now apply the functional equation of the Eisenstein series on $\Gamma_0(N)$ to relate $R^*(\phi, 1-s)$ to a sum of Rankin--Selberg transforms at the other cusps of $\Gamma_0(N)$. This takes the form
\begin{align} \label{eq:Eisenstein-functional-equation}
	R^*(\phi, 1-s ) = \sum_\mathfrak{a} \gamma_{\mathfrak{a}}(s)  R^*(\phi_\mathfrak{a}, s),
\end{align}
in which $\gamma_\mathfrak{a}(s)$ is an entry of the scattering matrix for $\Gamma_0(N)$ and $\phi_\mathfrak{a} = \phi\vert_{\sigma_{\mathfrak{a}}}$ under the weight $0$ slash operator. Exact formulas for $\gamma_\mathfrak{a}$ may be obtained by combining~\cite[Theorem 6.1]{Young19} and~\cite[Proposition 4.2]{Young19}. We have $\gamma_\mathfrak{a}(s) = O(1)$ in fixed vertical strips away from poles. By applying~\eqref{eq:Eisenstein-functional-equation} to~\eqref{eq:Rankin-Selberg-two-lines}, we conclude that
\[
	R_0^*(\phi, 1)
	= \frac{1}{2\pi i} \int_{(a)} \Big(\frac{g(s) R_0^*(\phi, s)}{s-1} 
			+ \frac{g(1-s)}{s} \sum_\mathfrak{a} \gamma_\mathfrak{a}(s)  R_0^*(\phi_\mathfrak{a}, s) \Big) ds.
\]

In our application, we take $\phi(z) = \phi_j(z) = y^k \vert f(z) \vert^2 \overline{\mu_j(z)}$, where $\mu_j$ is a Maass cusp form on $\Gamma_0(N)$ with $\Vert \mu_j \Vert = 1$. We conclude that
\begin{align} \label{eq:inner-product-vs-Rankin-Selberg-transform}
	\langle F, \mu_j \rangle
		= \frac{V_N}{12 i} \sum_\mathfrak{a}  \int_{(a)} 
			\Big( \delta_{[\mathfrak{a}=\infty]}s g(s) + (s-1) & g(1-s) \gamma_\mathfrak{a}(s)\Big) \\
				& \times \zeta^*(2s) R(\phi_{j\mathfrak{a}}, s) ds,
\end{align}
which generalizes~\cite[(2.10)]{Jutila96}. Note that this expression lets us determine $\langle F, \mu_j \rangle$ while only sampling $R(\phi_{j\mathfrak{a}},s)$ on the line $\Re s = a \gg 1$. We also note that the pole of $\zeta^*(2s)$ at $s= \frac{1}{2}$ is canceled by $R(\phi_{j\mathfrak{a}},s)$, hence the only poles of the integrand in $\Re s > 0$ are those of $R(\phi_{j\mathfrak{a}},s)$.

\begin{remark}
Following~\cite[(2.8)]{Jutila96}, we take $g(s) = \mathrm{exp}(1- \cos(\frac{s-1}{B}))$, for some large $B>0$. This choice implies $\vert g(s) \vert \ll \mathrm{exp}(-\frac{1}{2} \mathrm{exp}(\vert \Im s \vert / B))$ in the vertical strip $\vert \Re s - 1 \vert \leq \frac{\pi B}{3}$. In particular, the contour integral~\eqref{eq:inner-product-vs-Rankin-Selberg-transform} converges if $R(\phi_{j\mathfrak{a}},s)$ grows at most exponentially in $\vert \Im s \vert$. This will be established in Remark~\ref{rem:coarse-bound}. \hfill //
\end{remark}

To bound the Rankin--Selberg transform
\[
	R(\phi_{j\mathfrak{a}},s) 
		=
		\int_0^\infty \int_0^1 
			y^{s+k} \vert f_\mathfrak{a}(z) \vert^2 \overline{\mu_{j\mathfrak{a}}(z)} \frac{dxdy}{y^2},
\]
we represent $f_\mathfrak{a}$ and $\mu_{j\mathfrak{a}}$ as Fourier series, as described in~\eqref{eq:harmonic-Fourier-expansion-general} and~\eqref{eq:Maass-form-Fourier-expansion}, then execute the $x$-integral. This expresses $R(\phi_{j\mathfrak{a}},s)$ as a triple sum over integers $(n_1,n_2,n_3)$ subject to the relation $n_1 - n_2 = n_3$. As in section~\S{\ref{sec:unfolding}}, we group these terms based on the signs of $n_1$ and $n_2$, so that
\[
	R(\phi_{j \mathfrak{a}}, s) 
		= I^+_{j \mathfrak{a}}(s) + I^-_{j \mathfrak{a}}(s) 
			+ I^\times_{j \mathfrak{a}}(s) + I^0_{j \mathfrak{a}}(s),
\]
denoting the subsums in which $(n_1,n_2)$ are both positive, are both negative, have mixed sign, or contain a zero, respectively. By changing variables to introduce $m := \vert n_1-n_2 \vert = \vert n_3 \vert$ and grouping similar terms, we write
\begin{align}
	I_{j \mathfrak{a}}^+(s) \label{eq:I-plus-definition}
		&= \sum_{m,n+\varkappa_\mathfrak{a} > 0} 2\Re 
			\Big( c_\mathfrak{a}^+(n+m) \overline{c_\mathfrak{a}^+(n) \rho_{j\mathfrak{a}}(m)}\Big)
				\varphi_j^+(m,n+\varkappa_\mathfrak{a},s), \\
	I_{j \mathfrak{a}}^-(s)
		&= \sum_{m,n \geq 1} 2 \Re
			\Big( c_\mathfrak{a}^-(n) \overline{c_\mathfrak{a}^-(n+m) \rho_{j\mathfrak{a}}(m)} \Big)
				\varphi_j^-(m,n- \varkappa_\mathfrak{a},s), \\
	I_{j \mathfrak{a}}^\times(s) \label{eq:I-cross-definition}
		&= \sum_{m =1}^\infty \sum_{n = 1 - \lceil \varkappa_\mathfrak{a} \rceil}^{m-1} 2 \Re 
			\Big(c_\mathfrak{a}^+(n) \overline{c_\mathfrak{a}^-(m-n) \rho_{j\mathfrak{a}}(m)} \Big)
			\varphi_j^\times(m,n+\varkappa_\mathfrak{a},s),
\end{align}
in which the three functions $\varphi_j^+$, $\varphi_j^-$, and $\varphi_j^\times$ are defined by
\begin{align}
	\varphi_j^+(m,n,s) 
		&:= \int_0^\infty \! y^{s+k-\frac{1}{2}} e^{-2\pi (2n+m) y}
				K_{it_j}(2\pi m y) \frac{dy}{y}, \\ \label{eq:phi-definitions}
	\varphi_j^-(m,n,s)
		&:= \int_0^\infty \! y^{s+k-\frac{1}{2}} e^{2\pi (2n+m) y} \Gamma(1-k, 4\pi n y) \\[-5pt]
			& \hspace{30 mm} \times
			\Gamma(1-k, 4\pi(n+m) y)  K_{it_j}(2\pi m y) \frac{dy}{y}, \\
	\varphi_j^\times(m,n,s)
		&:= \int_0^\infty \! y^{s+k-\frac{1}{2}} e^{2\pi(m-2n)y} 
			\Gamma(1-k, 4\pi (m-n)y) K_{it_j}(2\pi m y) \frac{dy}{y}.
\end{align}
Here we have assumed without loss of generality that $\rho_{j\mathfrak{a}}(-m) = \overline{\rho_{j\mathfrak{a}}(m)}$ for Maass cusp forms of weight $0$. Lastly, for singular cusps, we define

\begin{align}
	I_{j\mathfrak{a}}^0(s) \label{eq:I-0-definition}
		&= \sum_{m > 0} 
			2 \Re \big( c_\mathfrak{a}^+(m) \overline{c_\mathfrak{a}^+(0) \rho_{j \mathfrak{a}}(m)} \big)
			\varphi_j^+(m,0,s) \\
		& \qquad +
		\sum_{m > 0} 
			2 \Re \big( c_\mathfrak{a}^+(m) \overline{c_\mathfrak{a}^-(0) \rho_{j \mathfrak{a}}(m)} \big)
			\varphi_j^+(m,0,s-k+1) \\
		& \qquad +
		\sum_{m > 0}
			2 \Re \big(c_\mathfrak{a}^-(m) \overline{c_\mathfrak{a}^+(0) \rho_{j \mathfrak{a}}(m)} \big)
			\varphi_j^\times(m,0,s) \\
		& \qquad +
		\sum_{m > 0}
			2 \Re \big(c_\mathfrak{a}^-(m) \overline{c_\mathfrak{a}^-(0) \rho_{j \mathfrak{a}}(m)} \big)
			\varphi_j^\times(m,0,s-k+1).
\end{align}
For non-singular cusps, we set $I_{j\mathfrak{a}}^0(s) = 0$, as the corresponding summands vanish or otherwise incorporate into $I_{j \mathfrak{a}}^+(s)$.

\begin{remark}
These decompositions mirror~\cite{Jutila96, Jutila97}, except that we separate $I_{j\mathfrak{a}}^+$ from $I_{j \mathfrak{a}}^-$ and introduce $I_{j \mathfrak{a}}^0$ to account for non-cuspidality. In fact, $\varphi_j^+$ exactly matches an unnamed function from~\cite[p.~449]{Jutila96}. Our functions $\varphi_j^-$ and $\varphi_j^\times$ can be viewed as variants of the functions $\varphi_j^+$ and $\varphi_j^-$ from~\cite[(3.4)]{Jutila97}, respectively. \hfill //
\end{remark}

\subsection{Representations and Estimates for \texorpdfstring{$\varphi_j^+$, $\varphi_j^\times$, and $\varphi_j^-$}{phi-plus, phi-cross, and phi-minus}}

In this section, we record some useful information about the functions $\varphi_j^+$, $\varphi_j^\times$, and $\varphi_j^-$. We first consider $\varphi_j^+$, leveraging earlier work of Jutila.

\begin{lemma}[{\cite[\S{3}]{Jutila96}}] \label{lem:phi-plus-behavior}
Define $\lambda = \lambda(m,n) := \sqrt{1-m^2/(2n+m)^2}$ and set $p := s+k-\frac{1}{2}$. The function $\varphi_j^+(m,n,s)$ defined in~\eqref{eq:phi-definitions} is analytic in $\Re p > 0$ and may be written in either of the forms
\begin{align}
	\varphi_j^+(m,n,s) \label{eq:phi-plus-representation-1}
	&= \frac{\sqrt{\pi} \, m^{it_j} \Gamma(p+it_j)\Gamma(p-it_j)}
					{(4\pi)^p (2n+m)^{p+it_j} \Gamma(p+\frac{1}{2})} \\
	& \qquad \qquad \qquad
		\times (1+\lambda)^{-p-it_j} {}_2F_1\Big(p, p+it_j, 2p
						\Big\vert \tfrac{2\lambda}{1+\lambda} \Big); \\
	\varphi_j^+(m,n,s)  \label{eq:phi-plus-representation-2}
	&= \frac{2^{-1-2p} \pi^{-p}}
					{(n(n+m))^{p/2}}  \\
	& \quad  
		\times \Big( \Big(\frac{1-\lambda}{1+\lambda}\Big)^{\frac{it_j}{2}} \Gamma(-it_j)\Gamma(p+it_j)
					{}_2F_1\Big(p, 1-p, 1+it_j
						\Big\vert \tfrac{\lambda-1}{2\lambda} \Big) \\
		& \qquad  
				+ \Big(\frac{1-\lambda}{1+\lambda}\Big)^{-\frac{it_j}{2}} \Gamma(it_j)\Gamma(p-it_j)
					{}_2F_1\Big(p, 1-p, 1-it_j
						\Big\vert \tfrac{\lambda-1}{2\lambda} \Big)\Big).
\end{align}
\end{lemma}

\begin{proof} These identities are implicit in~\cite[(3.16)-(3.21)]{Jutila96}.
\end{proof}

In the special case $n=0$, we have $\lambda = 0$ and~\eqref{eq:phi-plus-representation-1} implies that
\begin{align} \label{eq:phi-plus-closed-form}
	\varphi_j^+(m,0,s)
		= \frac{\sqrt{\pi} \, \Gamma(p+it_j)\Gamma(p-it_j)}
					{(4\pi m )^p  \Gamma(p+\frac{1}{2})},
\end{align}
which can also be seen directly via~\cite[6.621(3)]{GradshteynRyzhik07}. For $n \neq 0$, we don't expect simplification but can still produce upper bounds. For example, in the text surrounding~\cite[(4.5)]{Jutila97}, Jutila applies~\eqref{eq:phi-plus-representation-1} to produce 
\begin{align} \label{eq:phi-plus-upper-bound-1}
	\varphi_j^+(m,n,s)
		\ll_{\Re p} \frac{\vert \Gamma(p+it_j)\Gamma(p-it_j) \vert}
					{(2n+m)^{\Re p} (1+\lambda)^{\Re p} \vert \Gamma(p) \vert}
				\log(2n+m),
\end{align}
valid for $\Re p > 0$. An upper bound derived from the representation~\eqref{eq:phi-plus-representation-2} is presented in the following lemma.

\begin{lemma}[{cf.~\cite[p.452]{Jutila97}}]
Fix $t_j \in \mathbb{R}$ and $\epsilon > 0$. Suppose that $\lambda \neq 0$. For any $s$ in a fixed vertical strip away from poles,
\begin{align} \label{eq:phi-plus-upper-bound-2}
	\varphi_j^+(m,n,s) \ll_{\epsilon} 
		\frac{\vert t_j \vert^{\Re p - 1}}{(n(n+m))^{\frac{\Re p}{2}}}
	\Big( 1 + \Big\vert \frac{1+ \vert s\vert^2}{\lambda t_j} \Big\vert^{1+\vert \Re p \vert+\epsilon}\Big)
	\frac{e^{\frac{\pi}{2} \vert \Im s \vert}}{e^{\pi \vert t_j \vert}}. \qquad
\end{align}
\end{lemma}

\begin{proof} For $p \notin \mathbb{Z}$ and non-positive $z \in \mathbb{C}$, consider the integral representation
\begin{align} 
\begin{split} \label{eq:2F1-contour}
	& {}_2F_1\Big(p, 1-p, 1+it_j
						\Big\vert z \Big) \\
	& \qquad 
		= \int_B 
			\frac{\Gamma(1+it_j)\Gamma(p+w)\Gamma(1-p+w)\Gamma(-w)}{\Gamma(p)\Gamma(1-p)\Gamma(1+it_j + w)}
			(-z)^w \, dw,
\end{split}
\end{align}
in which the contour $B$ separates the poles of $\Gamma(p+w)\Gamma(1-p+w)$ from those of $\Gamma(-w)$~\cite[(15.6.6)]{DLMF}. We suppose that $\Re p  > 0$ and shift the contour $B$ to the line $\Re w = \Re p + \epsilon$. This shift passes poles and extracts residues at $w = 0,1,\ldots, \lfloor \Re p + \epsilon \rfloor$, totaling
\[
	\sum_{v=0}^{\lfloor \Re p + \epsilon \rfloor} 
		\frac{(p)_v (1-p)_v}{v! (1+it_j)_v} z^v
	\ll 1 + \Big\vert \frac{(1+ \vert p\vert^2) z}{t_j} \Big\vert^{\Re p + \epsilon},
\]
in which $(\alpha)_v := \Gamma(v+\alpha)/\Gamma(\alpha)$ denotes the Pochhammer symbol. The same upper bound holds for the shifted integral, by Stirling's approximation. We apply this estimate for $z = \frac{\lambda -1}{2\lambda} \ll \lambda^{-1}$, then apply Stirling's approximation to the other factors of~\eqref{eq:phi-plus-representation-2} to complete the case $\Re p > 0$.
The case $\Re p < 0$ then follows using the invariance of~\eqref{eq:2F1-contour} under $p \leftrightarrow 1-p$.
\end{proof}

We conclude our discussion of $\varphi_j^+$ by presenting a uniform upper bound for the size of its residues.

\begin{lemma} \label{lem:phi-plus-residue-bound}
Fix $t_j \in \mathbb{R}$. For each integer $r \geq 0$, we have
\[
	\Res_{s =\frac{1}{2}-k \pm it_j - r} \varphi_j^+(m,n,s)
	\ll_r (n+m)^r \vert t_j \vert^{-\frac{1}{2}} e^{-\frac{\pi}{2} \vert t_j \vert}.
\]
\end{lemma}

\begin{proof}
Stirling's approximation and~\eqref{eq:phi-plus-representation-2} give
\[
	\Res_{s = \frac{1}{2}-k + it_j - r} \varphi_j^+(m,n,s)
	\ll_r \frac{(n(m+n))^{\frac{r}{2}}}{\vert t_j \vert^{1/2} e^{\frac{\pi}{2} \vert t_j \vert}}
	\cdot \Big\vert {}_2F_1\Big(\begin{matrix} it_j-r, 1+r-it_j,\\ 1-it_j \end{matrix} \Big\vert \tfrac{\lambda-1}{2\lambda}\Big)\Big\vert.
\]
The transformation ${}_2F_1(a,b,c,z) = (1-z)^{-a} {}_2F_1(a,c-b,c, \frac{z}{z-1})$ (cf.~\cite[9.131(1)]{GradshteynRyzhik07}) relates the hypergeometric function above to the finite sum
\[
	\Big(\frac{\lambda+1}{2\lambda}\Big)^{-it_j + r}
	{}_2F_1\Big(\begin{matrix} it_j-r, -r,\\ 1-it_j \end{matrix} \Big\vert \tfrac{1-\lambda}{1+\lambda}\Big)
	\ll \lambda^{-r} \sum_{v=0}^r \frac{(it_j-r)_v (-r)_v}{(1-it_j)_v v!} \Big(\frac{1-\lambda}{1+\lambda}\Big)^v,
\]
which is $O_r(\lambda^{-r})$, uniformly in $t_j$. The claim now follows from the estimate $\lambda^2 \asymp \frac{n}{n+m}$, and the computation for $s=\frac{1}{2} -k-it_j -r$ is identical.
\end{proof}

To understand $\varphi_j^\times$ and $\varphi_j^-$, we express them as contour integral transforms of $\varphi_j^+$. The following lemma consolidates relevant information about $\varphi_j^\times$.

\begin{lemma} \label{lem:phi-cross-behavior}
The function $\varphi_j^\times(m,n,s)$ defined in~\eqref{eq:phi-definitions} admits meromorphic continuation to $s \in \mathbb{C}$, with poles at $s = -\frac{1}{2} \pm it_j -r$ and $s = \frac{1}{2} - k \pm it_j - r$, for $r \in \mathbb{Z}_{\geq 0}$. If $t_j \in \mathbb{R}$ and $\Re s > 0$ away from poles, we have
\begin{align} \label{eq:phi-cross-upper-bound}
	\varphi_j^\times(m,n,s)
	&\ll  \frac{1}{(m-n)^k\sqrt{m}}
	\Big( \frac{\vert s - it_j \vert \cdot \vert s + it_j\vert}{m \vert s \vert}\Big)^{\Re s -1}
	\vert s \vert^{-\frac{1}{2}} \\
	& \qquad
	\times \! \Big( 
		1 + \Big( 
					\frac{m\vert s \vert}{(m-n)\vert s - it_j \vert \cdot \vert s + it_j\vert}
					\Big)^{\vert k \vert +\epsilon} \Big)
	e^{-\pi \vert t_j \vert + \frac{\pi}{2} \vert \Im s \vert}  \\
	& \quad 
	+ \delta_{[\Re s < \frac{1}{2} - k]} \big(m+n+ \vert s \vert + \vert t_j \vert\big)^A
	e^{-2\pi \vert t_j \vert + \frac{3\pi}{2} \vert \Im s \vert}
\end{align}
for all $\epsilon > 0$ and for some $A>0$ depending only on $k$.
\end{lemma}

\begin{proof}
The integral representation~\eqref{eq:incomplete-gamma-integral-representation} implies that
\begin{align}
	\varphi_j^\times(m,n,s)
		&=
			\frac{-\pi}{\Gamma(k)} 
			\cdot  \frac{1}{2\pi i} \int_C \frac{\Gamma(w+k)}{(4\pi(m-n))^{w+k} \sin(\pi w)} \\
		& \qquad\qquad \qquad
	\times\Big( \int_0^\infty \! y^{s-\frac{1}{2}-w} e^{-2\pi m y} K_{it_j}(2\pi m y) \frac{dy}{y} \Big) dw \\
		\label{eq:phi-cross-integral-representation}
		&= 
			\frac{-\pi}{\Gamma(k)} \cdot  \frac{1}{2\pi i} 
			\int_C \frac{\Gamma(w+k)\varphi_j^+(m,0, s-k-w)}{(4\pi(m-n))^{w+k} \sin(\pi w)} dw,
\end{align}
where $C$ is a contour separating the poles of $\Gamma(w+k)$ from those at $w=0,1,\ldots$ arising from $1/\sin(\pi w)$. To begin, we require $\Re s > \frac{1}{2} + \max \{\Re w: w \in C\}$. To consider general $s$, we shift the contour $C$ left, passing poles from $\Gamma(w+k)$ and $1/\sin(\pi w)$ and extracting residues involving $\varphi_j^+(m,0,s)$ at shifted arguments. By~\eqref{eq:phi-plus-closed-form}, these residues contribute poles at the poles of $\Gamma(s+k-\frac{1}{2} \pm it_j)$ and $\Gamma(s + \frac{1}{2} \pm it_j)$.

To produce growth estimates, we then shift $C$ rightwards, to the contour $\Re w = \vert k \vert + \epsilon$. This extracts a sum of residues equal to
\begin{align*}
	& \sum_{q = 0}^{\vert k \vert - \frac{1}{2}} 
			\frac{(-1)^q\Gamma(q+k)}{\Gamma(k) (4\pi(m-n))^{k+q}} \cdot \varphi_j^+(m,0,s-k-q) \\[-.38em]
	& \qquad + \!\!\!\!
		\sum_{r=0}^{\lfloor \vert k \vert +\epsilon - \Re s + \frac{1}{2} \rfloor} \!\!\!\!
		\sum_{\pm}
			\Res_{w = s-\frac{1}{2} \pm it_j +r}
				\frac{\pi \Gamma(w+k)\varphi_j^+(m,0, s-k-w)}{(4\pi(m-n))^{w+k} \sin(\pi w) \Gamma(k)}.
\end{align*}
Stirling's approximation and Lemma~\ref{lem:phi-plus-residue-bound} show that the exponential decay in the residues in the second line is $e^{-\frac{3\pi}{2} \vert \Im s \pm it_j \vert - \frac{\pi}{2} \vert t_j \vert} \ll e^{-2\pi \vert t_j \vert + \frac{3\pi}{2} \vert \Im s \vert}$, while the worst polynomial growth is $O((m+n+ \vert s \vert + \vert t_j \vert)^A)$ for some $A> 0$ depending linearly on $\vert k \vert$ and $\Re s $. Since $\Re s \in (0, \vert k \vert + \frac{1}{2})$ when these terms appear, we may take the constant $A$ to depend on $k$ alone.

Exponential decay in $\vert \Im w \vert$ within the integrand bounds the shifted contour integral to at most a constant multiple of the integrand near $\vert k \vert + \epsilon$. Stirling's approximation and~\eqref{eq:phi-plus-closed-form} then complete the proof of~\eqref{eq:phi-cross-upper-bound}.
\end{proof}

The corresponding properties of $\varphi_j^-$ may be obtained in a similar (though more complicated) way and are summarized in the following lemma.

\begin{lemma}
The function $\varphi_j^-(m,n,s)$ defined in~\eqref{eq:phi-definitions} admits meromorphic continuation to $s \in \mathbb{C}$, with poles at $s = k- \frac{3}{2} - m \pm it_j$ and $s = \frac{1}{2} - k - m \pm it_j$, for $m \in \mathbb{Z}_{\geq 0}$. If $t_j \in \mathbb{R}$ and $\Re s > 3 \vert k \vert + 1$, then for all $\epsilon > 0$ we have
\begin{align} \label{eq:phi-minus-upper-bound}
	\varphi_j^-(m,n,s)
	&\ll \frac{\log (m+n)}{(n(m+n))^{k+\frac{1}{2}}}
	\Big( \frac{\vert s - it_j \vert \vert s + it_j\vert}{(2n+m)(1+\lambda)\vert s \vert }\Big)^{\Re s -k-1}\\
	& \qquad
	\times \Big(
		1 + \Big( \frac{(n+m)\vert s \vert^2}{n\vert s - it_j \vert^2 \vert s + it_j\vert^2}\Big)^{\vert k \vert + \epsilon} \Big)
		e^{-\pi \vert t_j \vert + \frac{\pi}{2} \vert \Im s \vert}.
\end{align}
\end{lemma}

\begin{proof}
Using~\eqref{eq:incomplete-gamma-integral-representation}, we write $\varphi_j^-(m,n,s)$ as a double contour integral,
\begin{align} \label{eq:phi-minus-double-integral}
	&\varphi_j^-(m,n,s)\\
		&\quad 
		= \frac{\pi^2}{\Gamma(k)^2 (2\pi i)^2} \int_{C_1} \int_{C_2}
		\frac{\Gamma(w_1+k) \Gamma(w_2+k) }
				{(4\pi n)^{w_1+k}(4\pi (n+m))^{w_2+k}\sin(\pi w_1) \sin(\pi w_2)}
		\\
		&\qquad \qquad\qquad\quad \times
		\Big(\int_0^\infty \! y^{s-k-\frac{1}{2}-w_1-w_2}
			e^{-2\pi (2n+m) y} K_{it_j}(2\pi m y) \frac{dy}{y}\Big) dw_2 dw_1 \\
		&\quad= \frac{\pi^2}{(2\pi i)^2} 
		\!\!\int_{C_1} \int_{C_2} \!\!\!\!
		\frac{\Gamma(w_1+k) \Gamma(w_2+k)\varphi_j^+(m,n,s-2k-w_1-w_2) dw_2 dw_1}
				{\Gamma(k)^2 (4\pi n)^{w_1+k}(4\pi (n+m))^{w_2+k}\sin(\pi w_1) \sin(\pi w_2)},
\end{align}
where $C_1$ and $C_2$ are instances of the contour $C$ described in Lemma~\ref{lem:phi-cross-behavior} and $\Re s > k +\frac{1}{2} + 2\max\{w : w \in C\}$ to begin. As in Lemma~\ref{lem:phi-cross-behavior}, shifting the contours left produces residues which determine the poles of $\varphi_j^-$. To produce growth estimates, we shift $C_1$ and $C_2$ to the lines $\Re w_1 = \Re w_2 = \vert k \vert + \epsilon$, extracting a series of single contour integrals and a double sum of residues from $1/\sin(\pi w_1)\sin(\pi w_2)$ equal to
\begin{align} \label{eq:phi-minus-residues}
	\sum_{q_1,q_2=0}^{\vert k \vert - \frac{1}{2}} \!
		\frac{(-1)^{q_1 + q_2} \Gamma(q_1+k)\Gamma(q_2 + k) \varphi_j^+(m,n,s-2k-q_1 - q_2)}
			{\Gamma(k)^2 (4\pi n)^{k+q_1} (4\pi(m+n))^{k+q_2}}.
\end{align}
Exponential decay in vertical strips implies that the contour integrals are bounded by their values near the near axis, whereby the bound~\eqref{eq:phi-plus-upper-bound-1} and Stirling's approximation gives~\eqref{eq:phi-minus-upper-bound}.
\end{proof}

\begin{remark} \label{rem:coarse-bound}
The upper bounds for $\varphi_j^+$, $\varphi_j^\times$, and $\varphi_j^-$ given in~\eqref{eq:phi-plus-upper-bound-1},~\eqref{eq:phi-cross-upper-bound}, and~\eqref{eq:phi-minus-upper-bound}, respectively, imply that $R(\phi_{j\mathfrak{a}},s)$ satisfies a bound of the form
\begin{align} \label{eq:unquantified-RS-bound}
	R(\phi_{j\mathfrak{a}},s) 
	\ll  (\vert s \vert + \vert t_j \vert)^A 
	e^{-\frac{\pi}{2}\vert t_j \vert + \frac{\pi}{2} \vert \Im s \vert},
\end{align}
for sufficiently large $\Re s$ and some $A>0$. Indeed, such a bound holds for each of $I_{j\mathfrak{a}}^+(s)$, $I_{j\mathfrak{a}}^\times(s)$, $I_{j\mathfrak{a}}^-(s)$, and $I_{j\mathfrak{a}}^0(s)$, by dyadic subdivision of their defining sums, polynomial growth bounds on $c_\mathfrak{a}^\pm(n)$, and a bound for $\rho_{j\mathfrak{a}}(n)$ such as Lemma~\ref{lem:Maass-form-coefficient-average}.

Note that~\eqref{eq:unquantified-RS-bound} implies that the contour integral~\eqref{eq:inner-product-vs-Rankin-Selberg-transform} for $\langle F, \mu_j \rangle$ converges for $\Re s$ sufficiently large. More specifically, it implies that $\langle F, \mu_j \rangle \ll \vert t_j \vert^A e^{-\frac{\pi}{2} \vert t_j \vert}$ for some $A>0$.
These coarse estimates also show that the integral in~\eqref{eq:inner-product-vs-Rankin-Selberg-transform} may be truncated to $\vert \Im s \vert = c \log(1+ \vert t_j \vert)$ for some $c>0$ while introducing negligible error. We assume this henceforth. \hfill //
\end{remark}

We conclude this section with an upper bound for $\varphi_j^-$ obtained via~\eqref{eq:phi-plus-upper-bound-2}. We assume $\lambda \gg \vert t_j \vert^{-1-\epsilon}$. We also assume that $\vert s \vert \ll \log \vert t_j \vert$, which holds without loss of generality by Remark~\ref{rem:coarse-bound}.  
The bound~\eqref{eq:phi-plus-upper-bound-2} implies that the contribution of the residues~\eqref{eq:phi-minus-residues} is
\[
	O\Big(
		\frac{\vert t_j \vert^{\Re s - k - \frac{3}{2} + \epsilon}}
					{(n(m+n))^{\frac{1}{2}\Re s  + \frac{k}{2} - \frac{1}{4}}}
		e^{-\pi \vert t_j \vert + \frac{\pi}{2} \vert \Im s \vert}
		\sum_{q_1,q_2}^{\vert k \vert - \frac{1}{2}} \Big(\frac{m+n}{n} \Big)^{\frac{q_1-q_2}{2}} \vert t_j \vert^{-q_1-q_2}
	\Big).
\]
Since $\lambda \asymp \sqrt{n}/\sqrt{n+m}$, the estimate $\lambda \gg \vert t_j \vert^{-1-\epsilon}$ implies that $\vert t_j \vert^{2+2\epsilon} \gg \frac{m+n}{n}$. Thus, up $\vert t_j \vert^\epsilon$-factors, the $(q_1,q_2)$-sum is dominated by the $q_1 = q_2 =0$ term. Our estimate for the $q_1=q_2=0$ term likewise acts as a bound for the shifted double contour and any of the single contour integrals associated to residues from $1/\sin(\pi w_1)$ or $1/\sin(\pi w_2)$.

The contribution of the residues from the poles of $\varphi_j^+(m,n,s-2k-w_1-w_2)$ (as either single contour integrals or residues from single contour integrals) is $O((m+n+\vert s \vert + \vert t_j \vert)^A e^{-2\pi \vert t_j \vert + \frac{3\pi}{2} \vert \Im s \vert})$ for some $A>0$, by Lemma~\ref{lem:phi-plus-residue-bound} and Stirling's approximation. (Note that at this level of precision it suffices to consider only the exponential factor in Stirling's approximation.) Thus
\begin{align} \label{eq:phi-minus-upper-bound-2}
	\varphi_j^-(m,n,s)
	& \ll \frac{\vert t_j \vert^{\Re s - k - \frac{3}{2} + \epsilon}}
				{(n(m+n))^{\frac{1}{2}\Re s  + \frac{k}{2} - \frac{1}{4}}}
		e^{-\pi \vert t_j \vert + \frac{\pi}{2} \vert \Im s \vert} \\
	& \quad 
		+ \delta_{[\Re s < \vert k - 1 \vert - \frac{1}{2}]}
		\big(m+n+\vert s \vert + \vert t_j \vert\big)^A
		e^{-2\pi \vert t_j \vert + \frac{3\pi}{2} \vert \Im s \vert}.
\end{align}

\subsection{Sum Truncation} \label{sec:sum-truncation}

For some $(m,n)$, the functions $\varphi_j^+$, $\varphi_j^\times$, and $\varphi_j^-$ may be made arbitrarily small by taking $\Re s$ very large. For example,~\eqref{eq:phi-cross-upper-bound} implies that $\varphi_j^\times(m,n,s)$ decays with respect to $\vert t_j \vert$ as $\Re s \to \infty$ provided $m< \vert t_j \vert^{2+\delta}$, for any fixed $\delta > 0$. In other words, we may truncate $I_{j \mathfrak{a}}^\times(s)$ to $m \ll \vert t_j \vert^{2+\delta}$ in our estimate for $\langle F, \mu_j \rangle$, with a negligible error. Likewise,~\eqref{eq:phi-plus-upper-bound-1} and~\eqref{eq:phi-cross-upper-bound} imply that $I_{j \mathfrak{a}}^0(s)$ may be truncated to $m \ll \vert t_j \vert^{2+\delta}$.

We claim that $I_{j\mathfrak{a}}^+(s)$ and $I_{j\mathfrak{a}}^-(s)$ may be truncated to $n(m+n) \ll \vert t_j \vert^{2+\delta}$ at the cost of negligible error. To prove this, we follow~\cite[(3.25)]{Jutila96} and subdivide cases based on whether $\lambda \ll \vert t_j \vert^{-1}$.
\begin{enumerate}
\item[a.] If $\lambda \ll \vert t_j \vert^{-1}$ and $n(m+n) \gg \vert t_j \vert^{2+\delta}$, then $\lambda^2 \ll \vert t_j \vert^{-2}$, so that $(2n+m)^2 \gg n(n+m) \vert t_j \vert^2$ after simplifying. The lower bound $n(m+n) \gg \vert t_j \vert^{2+\delta}$ implies that $2n+m \gg \vert t_j \vert^{2+\delta/2}$, hence $n+m \gg \vert t_j \vert^{2+\delta/2}$. In this case,~\eqref{eq:phi-plus-upper-bound-1} and~\eqref{eq:phi-minus-upper-bound} produce arbitrary polynomial improvements in $\vert t_j \vert$ as $\Re s \to \infty$.
\item[b.] If $\lambda \gg \vert t_j \vert^{-1}$ and $n(m+n) \gg \vert t_j \vert^{2+\delta}$, we instead argue using the upper bounds~\eqref{eq:phi-plus-upper-bound-2} and~\eqref{eq:phi-minus-upper-bound-2}.
\end{enumerate}

Let $J_{j\mathfrak{a}}^+$ and $J_{j\mathfrak{a}}^-$ denote the truncations of $I_{j\mathfrak{a}}^+$ and $I_{j\mathfrak{a}}^-$ to $n(m+n) \ll \vert t_j \vert^{2+\delta}$. Likewise, define $J_{j\mathfrak{a}}^\times$ and $J_{j\mathfrak{a}}^0$ as the truncations of $I_{j\mathfrak{a}}^\times$ and $I_{j\mathfrak{a}}^0$ to $m \ll \vert t_j \vert^{2+\delta}$.

\subsection{Estimation of the Truncated Sums}

To complete our estimation of the inner product $\langle F, \mu_j \rangle$, we bound the sums $J_{j \mathfrak{a}}^+(s)$, $J_{j \mathfrak{a}}^-(s)$, $J_{j \mathfrak{a}}^\times(s)$, and $J_{j \mathfrak{a}}^0(s)$ on the line $\Re s =\delta$, where $\delta$ is the same constant used to define the truncation conditions. We assume that $\vert \Im s \vert = O(\log \vert t_j \vert)$, by Remark~\ref{rem:coarse-bound}.

We first consider $J_{j \mathfrak{a}}^+(s)$, which we truncated to $n(m+n) \ll \vert t_j \vert^{2+\delta}$. We subdivide into dyadic intervals, with $m \sim M$ and $n \sim L$. On each dyadic subsum, we estimate $\varphi_j^+(m,n+\varkappa_\mathfrak{a},s)$ using~\eqref{eq:phi-plus-upper-bound-2}, which outperforms~\eqref{eq:phi-plus-upper-bound-1} in these regimes.
Since $L(L+M) \ll \vert t_j \vert^{2+\delta}$ and $\lambda^2 \asymp n/(n+m)$, we have $\lambda \gg \vert t_j \vert^{-1-\delta}$. This observation, and the free assumption $s = O(\log \vert t_j \vert)$, shows that~\eqref{eq:phi-plus-upper-bound-2} bounds a given dyadic sum by
\begin{align} \label{eq:phi-plus-dyadic}
	& \vert t_j \vert^{k - \frac{3}{2} + O(\delta)}
	e^{-\pi \vert  t_j \vert + \frac{\pi}{2} \vert \Im s \vert} 
	\sum_{\substack{m \sim M \\ n \sim L}} 
		\frac{\vert c_\mathfrak{a}^+(n+m) c_\mathfrak{a}^+(n) \rho_{j\mathfrak{a}}(m) \vert}
			{(L(L+M))^{\frac{k}{2}-\frac{1}{4}}}.
\end{align}

In the sum within~\eqref{eq:phi-plus-dyadic}, we apply Cauchy--Schwarz, Lemma~\ref{lem:Muller-coefficient-averages}, and Lemma~\ref{lem:Maass-form-coefficient-average} to compute
\begin{align}
	& \sum_{n \sim L} \vert c_\mathfrak{a}^+(n) \vert 
			\sum_{m \sim M} \vert c_\mathfrak{a}^+(n+m) \rho_{j \mathfrak{a}}(m) \vert \\
	&\qquad
	\ll \sum_{n \sim L} \vert c_\mathfrak{a}^+(n) \vert
		\Big(\sum_{q \sim M+L} \vert c_\mathfrak{a}^+(q) \vert^2 \Big)^{\frac{1}{2}}
		\Big(\sum_{m \sim M} \vert \rho_{j\mathfrak{a}}(m) \vert^2 \Big)^{\frac{1}{2}} \\
	&\qquad
	\ll L^{\frac{1}{2}} \big(L(L+M)^{k+ \vert k-1 \vert} \big)^{\frac{1}{2}}
		\big(M+\vert t_j \vert \big)^{\frac{1}{2}} e^{\frac{\pi}{2} \vert t_j \vert}.
\end{align}

Thus the $(L,M)$-dependence in~\eqref{eq:phi-plus-dyadic} is $L^\frac{1}{2} M^\frac{1}{2} (L(L+M))^{\frac{1}{4}+\frac{1}{2}\vert k-1 \vert}$ or $L^{\frac{1}{2}} (L(L+M))^{\frac{1}{4}+\frac{1}{2}\vert k-1 \vert}$.
In the first case, the dominant dyadic interval takes $M \sim \vert t_j \vert^2$ and $L \sim 1$, while in the second case we dominate by the $M \sim 1$ and $L \sim \vert t_j \vert$ subintervals (up to $\vert t_j \vert^\delta$ factors). Either way, we conclude that
\begin{align} \label{eq:J-plus-upper-bound}
	J_{j\mathfrak{a}}^+(s) \ll  
		\vert t_j \vert^{k + \vert k-1 \vert + O(\delta)}
		e^{-\frac{\pi}{2} \vert  t_j \vert + \frac{\pi}{2} \vert \Im s \vert}.
\end{align}

Our treatment of $J_{j\mathfrak{a}}^-(s)$ is essentially the same. We again subdivide into dyadic intervals, with $m \sim M$ and $n \sim L$, then apply~\eqref{eq:phi-minus-upper-bound-2}. The contribution from $(m+n+ \vert s \vert + \vert t_j \vert)^A e^{-2\pi \vert t_j \vert + \frac{3\pi}{2} \vert \Im s \vert}$ within $\varphi_j^-(m,n,s)$ is clearly $O( \vert t_j \vert^{A'} e^{-2\pi \vert t_j \vert})$ for some $A' > 0$, which will be exponentially non-dominant. Otherwise,~\eqref{eq:phi-minus-upper-bound-2} bounds a given dyadic interval by
\[
	\vert t_j \vert^{- k - \frac{3}{2} + O(\delta)} 
		e^{-\pi \vert t_j \vert + \frac{\pi}{2} \vert \Im s \vert}
		\sum_{\substack{m \sim M \\ n \sim L}} 
		\frac{\vert c_\mathfrak{a}^-(n+m) c_\mathfrak{a}^-(n) \rho_{j\mathfrak{a}}(m) \vert}
			{(L(L+M))^{\frac{k}{2}-\frac{1}{4}}},
\]
which matches the $J_{j\mathfrak{a}}^+(s)$ case except that we've multiplied by $\vert t_j \vert^{-2k}$ and replaced $c_\mathfrak{a}^+$ with $c_\mathfrak{a}^-$. By Lemma~\ref{lem:Muller-coefficient-averages}, the change $c_\mathfrak{a}^+ \mapsto c_\mathfrak{a}^-$ does not worsen our estimate. We conclude that
\begin{align} \label{eq:J-minus-upper-bound}
	J_{j\mathfrak{a}}^-(s) \ll 
		\vert t_j \vert^{\vert k-1 \vert-k + O(\delta)}
		e^{-\frac{\pi}{2} \vert  t_j \vert + \frac{\pi}{2} \vert \Im s \vert} .
\end{align}

We next consider $J_{j\mathfrak{a}}^\times(s)$. By applying~\eqref{eq:phi-cross-upper-bound} and disregarding the non-dominant contribution of $(m+n+ \vert s \vert + \vert t_j \vert)^A e^{-2\pi \vert t_j \vert + \frac{3\pi}{2} \vert \Im s \vert}$, we find that
\[
	J_{j\mathfrak{a}}^\times(s) \ll
		\vert t_j \vert^{-2+O(\delta)} e^{-\pi \vert t_j \vert + \frac{\pi}{2} \vert \Im s \vert} \!\!\!\!
		\sum_{m < \vert t_j \vert^{2+\delta}} \!\!\!\!
			\frac{\vert \rho_{j \mathfrak{a}}(m) \vert}{m^{-1/2}} \!\!\!\!
		\sum_{n=1-\lceil \varkappa_\mathfrak{a} \rceil}^{m-1} \!\!\!
		\frac{\vert c_\mathfrak{a}^+(n) c_\mathfrak{a}^-(m-n) \vert}{(m-n-\varkappa_\mathfrak{a})^k},
\]
under the standing assumptions on $s$. To estimate the sums, we map $n \mapsto m-n$ in the $n$-sum, restrict $m$ to a dyadic interval $m \sim M$, swap the order of summation, and apply Cauchy--Schwarz and Lemma~\ref{lem:Muller-coefficient-averages}:
\begin{align*}
	& \sum_{m \sim M} 
			\frac{\vert \rho_{j \mathfrak{a}}(m) \vert}{m^{-1/2}}
		\sum_{n=1}^{m}
		\frac{\vert c_\mathfrak{a}^+(m-n) c_\mathfrak{a}^-(n) \vert}{(n-\varkappa_\mathfrak{a})^k} \\
	& \qquad \ll
	M^{\frac{1}{2}} \sum_{n\leq 2M} \frac{\vert c_\mathfrak{a}^-(n) \vert}{n^k}
		\Big(\sum_{m \sim M} \vert \rho_{j\mathfrak{a}}(m) \vert^2\Big)^{\frac{1}{2}}
		\Big(\sum_{m \sim M} \vert c_\mathfrak{a}^+(m) \vert^2\Big)^{\frac{1}{2}} \\
	& \qquad \ll
	M^{\frac{1}{2} k + \frac{1}{2} \vert k -1 \vert + \frac{1}{2}}
		(M+ \vert t_j \vert)^{\frac{1}{2}} e^{\frac{\pi}{2} \vert t_j \vert} 
		\sum_{n\leq 2M} \frac{\vert c_\mathfrak{a}^-(n) \vert}{n^k}.
\end{align*}
The remaining $n$-sum has size $O(M^{\frac{1}{2}-\frac{1}{2} k + \frac{1}{2} \vert k -1 \vert} \log M)$ by dyadic subdivision, Cauchy--Schwarz, and Lemma~\ref{lem:Muller-coefficient-averages}.
The largest overall contribution to $J_{j\mathfrak{a}}^\times$ appears when $M \sim \vert t_j \vert^{2+\delta}$, which gives the estimate
\begin{align} \label{eq:J-cross-upper-bound}
	J_{j\mathfrak{a}}^\times(s) \ll
	e^{-\frac{\pi}{2} \vert t_j \vert + \frac{\pi}{2} \vert \Im s \vert}
	\begin{cases}
		\vert t_j \vert^{2k-1+O(\delta)}, & k > 1, \\
		\vert t_j \vert^{3-2k + O(\delta)}, & k <1.
	\end{cases}
\end{align}

Finally, we consider $J_{j\mathfrak{a}}^0(s)$, which we treat according to the four-part decomposition of $I_{j\mathfrak{a}}^0(s)$ in~\eqref{eq:I-0-definition}. Applying~\eqref{eq:phi-plus-upper-bound-1} and~\eqref{eq:phi-cross-upper-bound} and ignoring the non-dominant contribution of $(m+n+ \vert s \vert + \vert t_j \vert)^A e^{-2\pi \vert t_j \vert + \frac{3\pi}{2} \vert \Im s \vert}$ in the latter produces
\begin{align*}
	J_{j\mathfrak{a}}^0(s)
	\ll e^{-\pi \vert t_j \vert + \frac{\pi}{2} \vert \Im s \vert} \vert t_j \vert^{O(\delta)} \!\!\!
	& \sum_{m \ll \vert t_j \vert^{2+\delta}} \!\!
	\Big( \frac{\vert c_\mathfrak{a}^+(m) \rho_{j\mathfrak{a}}(m) \vert}
				{m^{k-\frac{1}{2}} \vert t_j \vert^{2-2k}} 
		+ \frac{\vert c_\mathfrak{a}^+(m) \rho_{j\mathfrak{a}}(m) \vert}{m^{\frac{1}{2}}} 
		 \\
	& \qquad \quad
	+ \frac{\vert c_\mathfrak{a}^-(m) \rho_{j\mathfrak{a}}(m) \vert}{m^{k-\frac{1}{2}} \vert t_j \vert^2} 
	+ \frac{\vert c_\mathfrak{a}^-(m) \rho_{j\mathfrak{a}}(m) \vert}{m^{\frac{1}{2}} \vert t_j \vert^{2k}} 
	\Big).
\end{align*}
For each term in the parenthetical, we subdivide dyadically on $m$, then apply Cauchy--Schwarz, Lemma~\ref{lem:Muller-coefficient-averages}, and Lemma~\ref{lem:Maass-form-coefficient-average}. In each term, the largest dyadic contribution has $m \sim \vert t_j \vert^{2+\delta}$. The first two terms contribute $O(\vert t_j \vert^{k+ \vert k-1 \vert + O(\delta)} e^{\frac{\pi}{2} \vert t_j \vert})$, while the last two are $O(\vert t_j \vert^{\vert k-1 \vert -k + O(\delta)} e^{\frac{\pi}{2} \vert t_j \vert})$. We conclude that
\begin{align} \label{eq:J-0-upper-bound}
	J_{j\mathfrak{a}}^0(s)
	\ll e^{-\frac{\pi}{2} \vert t_j \vert + \frac{\pi}{2} \vert \Im s \vert} 
	\cdot \begin{cases}
		\vert t_j \vert^{2k-1+O(\delta)}, & k > 1, \\
		\vert t_j \vert^{1+O(\delta)}, & k = \frac{1}{2}, \\
		\vert t_j \vert^{1-2k+O(\delta)}, & k < 0.
	\end{cases}
\end{align}

By combining the upper bounds derived in this section, we complete our estimation of $\langle F, \mu_j \rangle $ and prove Theorem~\ref{thm:triple-inner-product-bound}:

\begin{proof}[Proof of Theorem~\ref{thm:triple-inner-product-bound}]
We estimate~\eqref{eq:inner-product-vs-Rankin-Selberg-transform}, truncating the contour to $\vert \Im s \vert \leq c \log(1+ \vert t_j \vert)$ with negligible error by Remark~\ref{rem:coarse-bound}. We write $R(\phi_{j\mathfrak{a}},s) = I_{j\mathfrak{a}}^+(s) + I_{j\mathfrak{a}}^-(s) + I_{j\mathfrak{a}}^\times(s) + I_{j\mathfrak{a}}^0(s)$, truncating each term in the decomposition as described in section~\S{\ref{sec:sum-truncation}}. Within the truncated contour, we shift to $\Re s = \delta$ (with negligible error) and apply~\eqref{eq:J-plus-upper-bound},~\eqref{eq:J-minus-upper-bound},~\eqref{eq:J-cross-upper-bound}, and~\eqref{eq:J-0-upper-bound} to produce
\begin{align*}
	\langle F, \mu_j \rangle
	& \ll \sum_{\mathfrak{a}} e^{-\frac{\pi}{2} \vert t_j \vert} 
	\big( \vert t_j \vert^{2k-1+O(\delta)} + \vert t_j \vert^{3-2k+O(\delta)} \big) \\
	& \qquad \times \int_{(\delta)} 
			\big\vert \delta_{[\mathfrak{a}=\infty]}s g(s) + (s-1) g(1-s) \gamma_\mathfrak{a}(s)\big\vert 
			\cdot \vert \zeta^*(2s) \vert e^{\frac{\pi}{2} \vert \Im s \vert} ds.
\end{align*}
The integral is $O_{\mathfrak{a},\delta}(1)$, and the proof follows by taking $\delta$ near $0$.
\end{proof}


\section{Bounding \texorpdfstring{$D_\ell(s)$}{Dh(s)} in Vertical Strips, Part II} \label{sec:D_h-growth-II}

In section~\S{\ref{sec:D_h-growth-bounds}}, we proved Proposition~\ref{prop:D_h-preliminary-growth}, which reduced the problem of bounding $D_\ell(s)$ to the problem of bounding $\langle \mathcal{V}_\mathcal{H}, P_\ell(\cdot, \overline{s}) \rangle$. In this section, we estimate the latter to prove the following theorem.

\begin{theorem} \label{thm:D_h-growth}
Fix $\epsilon > 0$ small. In the vertical strip $\Re s \in (\frac{1}{2}+\epsilon,\frac{3}{2}+\epsilon)$ away from poles of $D_\ell(s)$, we have
\[
	D_\ell(s) 
	\ll_{\epsilon} \ell^\epsilon \vert s \vert^\epsilon \cdot
	\big(\vert s \vert^{\frac{5}{2}} + \ell^{\frac{1}{4}} \vert s \vert^{2}
		+ \ell \vert s \vert^{-\frac{3}{2}} \big)^{\frac{3}{2} - \Re s}.
\]
\end{theorem}

Our proof of Theorem~\ref{thm:D_h-growth} follows the decomposition of $\langle \mathcal{V}_\mathcal{H}, P_\ell(\cdot,\overline{s}) \rangle$ into discrete and continuous spectra.

\subsection{Growth of the Discrete Spectrum \texorpdfstring{$\Sigma_{\mathrm{disc}}$}{}}

For convenience, recall that the discrete spectrum equals
\[
\Sigma_{\mathrm{disc}}(s) 
		:=  \frac{\sqrt{\pi}}{(4\pi \ell)^{s-\frac{1}{2}} \Gamma(s)}
		\sum_j \rho_j(\ell) \Gamma(s - \tfrac{1}{2} + it_j)\Gamma(s- \tfrac{1}{2} -it_j)
			\langle \mathcal{V}_\mathcal{H}, \mu_j \rangle.
\]
By the comments at the start of section~\S{\ref{sec:jutila-triple-inner-products}}, we may replace $\mathcal{V}_\mathcal{H}$ here with the unregularized form $y^{3/2} \vert \mathcal{H}(z) \vert^2$. Then, by Theorem~\ref{thm:triple-inner-product-bound} and Stirling,
\begin{align*}
	\Sigma_{\mathrm{disc}}(s)
	& \ll \ell^{\frac{1}{2}-\Re s} \vert s \vert^{\frac{1}{2}-\Re s}
		e^{\frac{\pi}{2} \vert \Im s \vert}
	\sum_j \frac{\vert \rho_j(\ell)\vert}{\cosh \frac{\pi}{2} t_j} \cdot \vert t_j \vert^{2+\epsilon} \\
	& \qquad \qquad \times 
	\vert s + it_j \vert^{\Res - 1} \vert s - it_j \vert^{\Re s -1}
	e^{-\pi \max(\vert t_j \vert, \vert \Im s \vert)}.
\end{align*}
Here we have used that $t_j \in \mathbb{R}$ for Maass forms on $\Gamma_0(4)$.

By Lemma~\ref{lem:Kuznetsov-long}, the mass in the $t_j$-sum in $\Sigma_{\mathrm{disc}}(s)$ concentrates to within $\vert t_j \vert < \vert \Im s \vert$. Thus
\begin{align} \label{eq:concentrated-disc}
	\Sigma_{\mathrm{disc}}(s)
	& \ll \ell^{\frac{1}{2}-\Re s} 
	\frac{\vert s \vert^{\frac{5}{2}-\Re s + \epsilon}}{e^{\frac{\pi}{2} \vert \Im s \vert}} \\
	& \qquad \times 
	\sum_{\vert t_j \vert < \vert \Im s \vert}
		\frac{\vert \rho_j(\ell)\vert}{\cosh \frac{\pi}{2} t_j} 
	\vert s + it_j \vert^{\Res - 1} \vert s - it_j \vert^{\Re s -1}.
\end{align}
Lemma~\ref{lem:Kuznetsov-long} implies a short-interval second moment estimate of the form
\begin{align} \label{eq:short-interval-kuznetsov}
	\sum_{X \leq \vert t_j \vert \leq X+1} 
	\frac{\vert \rho_{j}(\ell) \vert^2}{\cosh \pi t_j }
	\ll_{N,\epsilon} X^{1+\epsilon} \ell^\epsilon + \ell^{\frac{1}{2}+\epsilon}.
\end{align}
By dividing the range of summation in~\eqref{eq:concentrated-disc} into sub-intervals of length $1$ and applying Cauchy--Schwarz and~\eqref{eq:short-interval-kuznetsov} to each sub-interval, we find
\begin{align} \label{eq:discrete-bound}
	\Sigma_{\mathrm{disc}}(s)
	& \ll_{\Re s,\epsilon} \ell^{\frac{1}{2}-\Re s+\epsilon}
	\vert s \vert^{2+\epsilon} 
	\big( \vert s \vert^{\Re s} + 1 \big)
	\big(\vert s \vert^{\frac{1}{2}} + \ell^{\frac{1}{4}}\big)
	e^{-\frac{\pi}{2} \vert \Im s \vert}.
\end{align}

\subsection{Growth of the Continuous Spectrum \texorpdfstring{$\Sigma_{\mathrm{cont}}$}{}}

For convenience, recall that the continuous spectrum equals
\[
\Sigma_{\mathrm{cont}}
		=  \frac{V_N}{2} \sum_\mathfrak{a} \!\!
		\int_{-\infty}^\infty  \!\!
			\frac{\varphi_{\mathfrak{a}\ell}(\frac{1}{2}+it) \Gamma(s-\frac{1}{2}+it)\Gamma(s-\frac{1}{2}-it)}
				{(4\pi \ell)^{s-\frac{1}{2}} (\pi \ell)^{-it} \Gamma(s)\Gamma(\frac{1}{2}+it)}
			\langle \mathcal{V}_\mathcal{H}, E_\mathfrak{a}(\cdot, \tfrac{1}{2}+it) \rangle dt
\]
in $\Re s > \frac{1}{2}$.
To bound the growth of $\Sigma_{\mathrm{cont}}(s)$ with respect to $\vert \Im s \vert$ in this region, we must control the growth of both $\varphi_{\mathfrak{a}\ell}(\frac{1}{2}+it)$ and $\langle \mathcal{V}_\mathcal{H}, E_\mathfrak{a}(\cdot, \tfrac{1}{2}+it) \rangle$.
Sufficient estimates for $\varphi_{\mathfrak{a}\ell}(\frac{1}{2}+it)$ appear in~\eqref{eq:Eisenstein-Fourier-coefficient-bound}.

To estimate $\langle \mathcal{V}_\mathcal{H}, E_\mathfrak{a}(\cdot, \tfrac{1}{2}+it) \rangle$, we apply the Phragm\'{e}n--Lindel\"{o}f convexity principle to $\langle \mathcal{V}_\mathcal{H}, E_\mathfrak{a}(\cdot, \overline{w}) \rangle$, studying the latter outside the critical strip. We prove the following result.
	
\begin{proposition} \label{prop:V_H-Rankin-Selberg-bound}
For all $\epsilon > 0$, $\langle \mathcal{V}_\mathcal{H}, E_\mathfrak{a}(\cdot, \frac{1}{2} + it) \rangle \ll_{\epsilon} (1 + \vert t \vert)^{\frac{5}{2}+\epsilon} e^{-\frac{\pi}{2} \vert t \vert}$.
\end{proposition}
	
\begin{proof}
To begin, we interpret $\langle \mathcal{V}_\mathcal{H}, E_\mathfrak{a}(\cdot, \overline{w}) \rangle$ via the Rankin--Selberg method. More precisely, we interpret the inner product using Zagier's extension of the Rankin--Selberg method to functions with polynomial growth at cusps, as generalized to congruence subgroups by Gupta~\cite{ZagierRankinSelberg,Gupta00}.

Recall from~\eqref{eq:regularization-3/2} that $\mathcal{V}_\mathcal{H}(z)$ differs from $y^{3/2} \vert \mathcal{H}(z)\vert^2$ by a linear combination of the functions $E_\mathfrak{b}(z, \frac{3}{2})$, $\widetilde{E}_\mathfrak{b}(z,1)$, and $y^{1/2} \vert \theta(z) \vert^2$. It follows that
\[
	\mathcal{V}_\mathcal{H}(\sigma_{\mathfrak{a}} z)
	= \psi_{\mathfrak{a}}(y)  + O(y^{-M})
\]
for all $M>0$ as $y \to \infty$, in which $\psi_\mathfrak{a}(y)$ is a linear combination of $y^{-1/2}$ (from $E_\mathfrak{b}(z,\frac{3}{2})$), $\log y$, and $y^0$ (both from $\widetilde{E}_\mathfrak{b}(z,1)$). We define the Rankin--Selberg transform $R_\mathfrak{a}(\mathcal{V}_\mathcal{H}, w)$ by
\begin{align} \label{eq:Rankin-Selberg-transform}
	R_\mathfrak{a}(\mathcal{V}_\mathcal{H}, w)
		:= \int_0^\infty \int_0^1
		y^w \big( \mathcal{V}_\mathcal{H}(\sigma_\mathfrak{a} z) - \psi_\mathfrak{a}(y) \big) \frac{dxdy}{y^2}.
\end{align}

We write $\mathcal{V}_\mathcal{H}(z)$ as a Fourier series and execute the $x$-integral in~\eqref{eq:Rankin-Selberg-transform}, extracting the constant Fourier coefficient. This produces
\begin{align} \label{eq:Rankin-Selberg-transform-unfolded}
	R_\mathfrak{a}(\mathcal{V}_\mathcal{H}, w)
		& := \int_0^\infty y^{w+\frac{1}{2}} 
		\sum_{n +\varkappa_\mathfrak{a} > 0} \vert c_\mathfrak{a}^+(n) \vert^2 
		e^{-4\pi (n+\varkappa_\mathfrak{a}) y} \frac{dy}{y} \\
		& \quad +
		\int_0^\infty y^{w+\frac{1}{2}} \sum_{n \geq 1} \vert c_\mathfrak{a}^-(n) \vert^2
		\Gamma(-\tfrac{1}{2},4\pi (n-\varkappa_\mathfrak{a}) y)^2 
		e^{4\pi (n-\varkappa_\mathfrak{a}) y} \frac{dy}{y} \\
		& \quad - \frac{1}{64\pi^2} 
		\int_0^\infty y^{w-\frac{1}{2}}
			\sum_{n +\varkappa_\mathfrak{a} > 0} \vert r_\mathfrak{a}(n) \vert^2 e^{-4\pi (n + \varkappa_\mathfrak{a}) y} \frac{dy}{y},
\end{align}
where $\theta\vert_{\sigma_{\mathfrak{a}}}(z) = \sum_{n \geq 0} r_\mathfrak{a}(n) e((n+\varkappa_\mathfrak{a})z)$. Note that the constant Fourier coefficients of $E_\mathfrak{b}(z,\frac{3}{2})$ and $\widetilde{E}_\mathfrak{b}(z,1)$ cancel with corresponding terms in $\psi_{\mathfrak{a}}(y)$ and do not appear above. It follows that
\begin{align*}
	R_\mathfrak{a}(\mathcal{V}_\mathcal{H}, w)
		& = \frac{\Gamma(w+\frac{1}{2})}{(4\pi)^{w+\frac{1}{2}}}
		\sum_{n > -\varkappa_\mathfrak{a}} \!
				\frac{\vert c_\mathfrak{a}^+(n) \vert^2}{(n+\varkappa_\mathfrak{a})^{w+\frac{1}{2}}} 
		- \frac{\Gamma(w-\frac{1}{2})}{4 (4\pi)^{w+\frac{3}{2}}} 
			\sum_{n > -\varkappa_\mathfrak{a}} \!
				\frac{\vert r_\mathfrak{a}(n) \vert^2}{(n+\varkappa_\mathfrak{a})^{w-\frac{1}{2}}}  \\
		& \qquad 
		+ \sum_{n \geq 1} 
		\frac{\vert c_\mathfrak{a}^-(n) \vert^2}{(4\pi(n-\varkappa_\mathfrak{a}))^{w+\frac{1}{2}}}
		\int_0^\infty y^{w+\frac{1}{2}} \Gamma(-\tfrac{1}{2}, y)^2 e^{y} \frac{dy}{y}.
\end{align*}

Lemma~\ref{lem:Muller-coefficient-averages} implies that the two Dirichlet series converge in $\Re w > \frac{3}{2}$. Note that the integral above equals $G_{3/2}(w,1,1)$ as defined in~\eqref{eq:I-minus-definition}, so by the comments following~\eqref{eq:I-minus-definition}, the second line above converges for $\Re w > \frac{3}{2}$.

To estimate the growth of $R_\mathfrak{a}(\mathcal{V}_\mathcal{H}, w)$ on the line $\Re w = \frac{3}{2} + \epsilon$, we must quantify the growth of $G_{3/2}(w,1,1)$ with respect to $\vert \Im w \vert$. This was computed in Lemma~\ref{lem:G_k-growth}; away from poles, we have
\[
	G_\frac{3}{2}(w,1,1)  \ll_\epsilon \vert w \vert^{\Re w - \frac{3}{2}+\epsilon} e^{-\frac{\pi}{2} \vert \Im w \vert}.
\]
It follows that $R_\mathfrak{a}(\mathcal{V}_\mathcal{H}, w) \ll \vert w \vert^{\frac{3}{2}+\epsilon} e^{-\frac{\pi}{2} \vert \Im w \vert}$ on the line $\Re w = \frac{3}{2} + \epsilon$.

The estimate $\zeta^*(2-2w) R_\mathfrak{a}(\mathcal{V}_\mathcal{H}, 1-w) \ll\sum_\mathfrak{b} \zeta^*(2w) R_\mathfrak{b}(\mathcal{V}_\mathcal{H}, w)$ (cf.~\eqref{eq:Eisenstein-functional-equation}) can be used to produce bounds in a left half-plane. In particular, we find $R_\mathfrak{a}(\mathcal{V}_\mathcal{H}, w) \ll \vert w \vert^{\frac{7}{2} + \epsilon} e^{-\frac{\pi}{2} \vert \Im w \vert}$ on $\Re w = - \frac{1}{2} - \epsilon$. The Phragm\'{e}n--Lindel\"{o}f convexity principle then implies
\[
	R_\mathfrak{a}(\mathcal{V}_\mathcal{H}, \tfrac{1}{2} +  i t)
	\ll (1+\vert t \vert)^{\frac{5}{2} + \epsilon} e^{-\frac{\pi}{2} \vert t \vert}.
\]
for real $t$. To complete the proof, we note that $R_\mathfrak{a}(\mathcal{V}_\mathcal{H}, w) = \langle \mathcal{V}_\mathcal{H}, E_\mathfrak{a}(\cdot, \overline{w}) \rangle$ within the critical strip $\Re w \in (0,1)$ by~\cite[Proposition A.3]{HKLDWSphere}. (The constant $\Theta$ defined therein equals $0$, since $\psi_\mathfrak{a}(y)$ is a linear combination of $\log y$, $y^0$, and $y^{-1/2}$ for each $\mathfrak{a}$.)
\end{proof}

By Proposition~\ref{prop:V_H-Rankin-Selberg-bound},~\eqref{eq:Eisenstein-Fourier-coefficient-bound}, and Stirling's approximation, we have
\begin{align*}
	\Sigma_{\mathrm{cont}}(s)
	 \ll \frac{\ell^{\frac{1}{2}-\Re s+\epsilon} e^{\frac{\pi}{2} \vert \Im s \vert}}{\vert s \vert^{\Re s - \frac{1}{2}}} \!\!
	 \int_{-\infty}^\infty \!\!
		\vert s + it \vert^{\Re s -1} \vert s - it \vert^{\Re s -1}
		\frac{(1+ \vert t \vert)^{\frac{5}{2}+\epsilon}}{e^{\pi \max(\vert \Im s \vert, \vert t \vert)}} dt.
\end{align*}
The mass of the integral above concentrates in $\vert t \vert < \vert \Im s \vert$; restricting to this range, we find that
\begin{align}
	\Sigma_{\mathrm{cont}}(s)
	 & \ll \frac{\ell^{\frac{1}{2}-\Re s+\epsilon} \vert s \vert^{3 - \Re s + \epsilon}}{e^{\frac{\pi}{2} \vert \Im s \vert}}
		\int_{-\vert \Im s \vert}^{\vert \Im s \vert}
		\vert s + it \vert^{\Re s -1} \vert s - it \vert^{\Re s -1} dt \\ \label{eq:continuous-bound}
	& \ll 
	\frac{\ell^{\frac{1}{2}-\Re s+\epsilon} \vert s \vert^{2 + \epsilon}}{e^{\frac{\pi}{2} \vert \Im s \vert}}
	\big( \vert s \vert^{\Re s} + 1\big), \hspace{-2 mm} 
\end{align}
at least in the region $\Re s > \frac{1}{2}$ (where $\Sigma_{\mathrm{cont}}$ has this one-term description).

\subsection{Growth of \texorpdfstring{$D_\ell(s)$}{Dh}}

The upper bounds provided in~\eqref{eq:discrete-bound} and~\eqref{eq:continuous-bound} imply that $\Sigma_{\mathrm{cont}}(s) \ll \Sigma_{\mathrm{disc}}(s)$ in $\Re s > \frac{1}{2}$. It follows that
\begin{align} \label{eq:spectral-final-bound}
	\frac{\langle \mathcal{V}_\mathcal{H}, P_\ell(\cdot, \overline{s}) \rangle}{\Gamma(s+\frac{1}{2})}
	\ll_\epsilon \ell^{\frac{1}{2} - \Re s + \epsilon} \vert s \vert^{2+\epsilon}
	\big( \vert s \vert^{\frac{1}{2}} + \ell^{\frac{1}{4}} \big)
\end{align}
in this region. By combining this estimate with Proposition~\ref{prop:D_h-preliminary-growth} and the convexity principle, we complete our proof of Theorem~\ref{thm:D_h-growth}.

\begin{proof}[Proof of Theorem~\ref{thm:D_h-growth}]
For $\Re s > \frac{3}{2}$, the upper bound
\[
	D_\ell(s) 
	\ll \Big(\sum_{n \geq 1} \frac{H(n)^2}{(n+\ell)^{\Re s + \frac{1}{2}}}\Big)^\frac{1}{2}
		\Big(\sum_{n \geq 1} \frac{H(n+\ell)^2}{(n+\ell)^{\Re s + \frac{1}{2}}}\Big)^\frac{1}{2}
	\ll \sum_{n \geq 1} \frac{H(n)^2}{n^{\Re s + \frac{1}{2}}}
	\ll 1
\]
implies that the result holds on $\Re s = \frac{3}{2} + \epsilon$, for $\epsilon > 0$. The result also holds on the line $\Re s = \frac{1}{2} + \epsilon$, by Proposition~\ref{prop:D_h-preliminary-growth} and~\eqref{eq:spectral-final-bound}. The full theorem now follows by the convexity principle.
\end{proof}


\section{Applying a Truncated Perron Formula} \label{sec:perron}

To prove our main arithmetic result, Theorem~\ref{thm:intro-main-theorem}, we apply a truncated Perron formula to $D_\ell(s)$. Fix $\epsilon > 0$. For $X$ non-integral, we have
\begin{align} \label{eq:perron}
	\sum_{n \leq X} H(n) H(n-\ell)
		& = \frac{1}{2\pi i} \int_{2+\epsilon-i T}^{2+\epsilon+ i T} D_\ell(s-\tfrac{1}{2}) \frac{X^s}{s} ds \\
		& \quad
		+ O\Big(\frac{X^{2+\epsilon}}{T}
		+ \!\! \sum_{n=X/2}^{2X} \!\!
			\vert H(n)H(n-\ell)\vert \min\big(1,\tfrac{X}{T \vert X-n \vert}\big)\Big)
\end{align}
by~\cite[Corollary 5.3]{MontgomeryVaughn06}. By Lemma~\ref{lem:H-growth}, the error term in~\eqref{eq:perron} is
\[
	O\Big( \frac{X^{2+\epsilon}}{T} + 
	X^{1+\epsilon} \sum_{n=X/2}^{2X} \min\big(1,\tfrac{X}{T \vert X-n \vert}\big)\Big)
	= O\Big( \frac{X^{2+\epsilon}}{T} \Big).
\]

To estimate the integral in~\eqref{eq:perron}, we shift the contour from $\Re s = 2+\epsilon$ to $\Re s = 1+\epsilon$. By Theorem~\ref{thm:D_h-poles}, this extracts two residues, which total
\[
	\tfrac{1}{2} X^2 \Res_{s=\frac{3}{2}} D_\ell(s)
	+ \tfrac{2}{3} X^\frac{3}{2} \Res_{s=1} D_\ell(s).
\]
Shifting the truncated contour introduces error terms from horizontal contour integrals, which by Theorem~\ref{thm:D_h-growth} are bounded by
\begin{align*}
	& O\Big(
	\int_{1+i T+\epsilon}^{2+i T + \epsilon} \!\!
	D_\ell(s-\tfrac{1}{2}) \frac{X^s}{s} ds
	\Big)
	\ll
	\frac{(\ell T)^\epsilon}{T} \!
	\int_{1+\epsilon}^{2+ \epsilon} \!
	\big(T^{\frac{5}{2}} + \ell^\frac{1}{4} T^{2} + \ell T^{-\frac{3}{2}} \big)^{2-\sigma} 
		X^{\sigma} d \sigma \\
	& \qquad \ll (\ell XT)^{\epsilon'}\Big(\frac{X^2}{T} + X T^{\frac{3}{2}} 
			+ \ell^{\frac{1}{4}} X T + \ell  X T^{-\frac{5}{2}}  \Big).
\end{align*}

Once the contour is shifted to $\Re s = 1 + \epsilon$, we separate the contribution of the discrete spectrum $\Sigma_{\mathrm{disc}}(s)$ from the rest of $D_\ell(s)$. The estimates from Proposition~\ref{prop:D_h-preliminary-growth} and~\eqref{eq:continuous-bound} imply that the non-$\Sigma_{\mathrm{disc}}$ terms contribute
\begin{align*}
	O\Big(\int_{1-iT+\epsilon}^{1+iT+\epsilon} (\ell \vert s \vert)^\epsilon 
		\Big(  \vert s \vert^2 
			 +  \frac{\ell^{\frac{1}{2}}}{\vert s \vert} 
			 + \frac{\ell}{\vert s \vert^{\frac{3}{2}}} \Big) 
			\frac{X^{1+\epsilon}}{\vert s \vert} ds \! \Big) 
	\ll (\ell X T)^\epsilon \big( X \ell   + X T^2 \big).
\end{align*}

To bound the contribution of $\Sigma_{\mathrm{disc}}(s-\frac{1}{2})/\Gamma(s)$, we shift the contour farther left, to $\Re s = \epsilon$. This shift introduces an error term (from the horizontal contours), which has size
\[
	O\Big(
	(\ell X T)^\epsilon \cdot
	\big( T^{\frac{3}{2}}X + \ell^{\frac{1}{4}} T X 
			+ T^2 + \ell^{\frac{1}{4}} T^{\frac{3}{2}} \big)
	\Big),
\]
by~\eqref{eq:discrete-bound} as well as a finite sum of residues equal to
\[
	\mathfrak{R} :=
		\sum_{\vert t_j \vert < T}
		\Big( \frac{X^{1+it_j}}{\Gamma(2+it_j)} \Res_{s=\frac{1}{2}+it_j} \Sigma_{\mathrm{disc}}(s) 
		+  \frac{X^{1-it_j}}{\Gamma(2-it_j)} \Res_{s=\frac{1}{2}-it_j} \Sigma_{\mathrm{disc}}(s) \Big).
\]
The contribution of $\Sigma_{\mathrm{disc}}$ on the contour $\Re s = \epsilon$ is $O( (\ell X T)^\epsilon \cdot ( \ell T^3 + \ell^{\frac{5}{4}} T^{\frac{5}{2}} ))$ by~\eqref{eq:discrete-bound}. Evaluating the residues in $\mathfrak{R}$ and bounding in absolute values gives
\begin{align*}
	\mathfrak{R} & 
	\ll X
	\sum_{\vert t_j \vert < T}
		\frac{\vert \rho_j(\ell) \langle \mathcal{V}_\mathcal{H}, \mu_j \rangle \vert}{\vert t_j \vert^2} \\
&	\ll X T^\epsilon \sum_{\vert t_j \vert < T} 
			\frac{\vert \rho_j(\ell)\vert}{\cosh \frac{\pi}{2} t_j}
	\ll X T^{1+\epsilon} \Big( \sum_{\vert t_j \vert < T}  
	\frac{\vert \rho_j(\ell) \vert^2}{\cosh \pi t_j}  \Big)^{\frac{1}{2}},
\end{align*}
in which we've applied Theorem~\ref{thm:triple-inner-product-bound} and Cauchy--Schwarz. Lemma~\ref{lem:Kuznetsov-long} then implies that $\mathfrak{R} \ll_\epsilon X (\ell T)^\epsilon (T^2+ \ell^{\frac{1}{4}} T^{\frac{3}{2}})$.

Putting everything together and omitting obviously non-dominant errors, we conclude that
\begin{align*}
	& \sum_{n \leq X} H(n) H(n-\ell)  =
		\tfrac{1}{2} X^2 \Res_{s=\frac{3}{2}} D_\ell(s)
		+ \tfrac{2}{3} X^\frac{3}{2} \Res_{s=1} D_\ell(s) \\
	 & \qquad\qquad\qquad
		+ O_{\epsilon}\Big( \! (\ell X T)^\epsilon 
		\Big(\frac{X^2}{T} 
						+ X\big(T^2 + \ell^{\frac{1}{4}} T^{\frac{3}{2}} 
						+ \ell \big) 
						+ \ell T^3 + \ell^{\frac{5}{4}} T^{\frac{5}{2}}\Big) \!\Big).
\end{align*}
When $\ell \ll X^{2/3}$, these errors are minimized by setting $T = X^{1/3}$, producing a collected error of size $O(X^{\frac{5}{3}+\epsilon})$. In the range $X^{2/3} \ll \ell \ll X$, we choose any $T \in [X/\ell, X^{\frac{2}{5}} \ell^{-\frac{1}{10}}]$, producing a collected error of size $O(X^{1+\epsilon} \ell)$.
Using the residue formulas from Theorem~\ref{thm:D_h-poles}, we conclude that
\begin{align*}
	& \sum_{n \leq X} H(n) H(n-\ell) \\[-.5em]
	& \qquad 
		= \frac{\pi^2 X^2}{252\zeta(3)}
				\big(2\sigma_{-2}(\tfrac{\ell}{4}) - \sigma_{-2}(\tfrac{\ell}{2})+\sigma_{-2}(\ell_o) \big) 
	+ O_\epsilon\big(X^{\frac{5}{3}+\epsilon} 
			+ X^{1+\epsilon} \ell \big).
\end{align*}
Theorem~\ref{thm:intro-main-theorem} then follows by assuming $\ell \ll X$ and mapping $X \mapsto X+ \ell$.

\begin{remark} \label{rem:smooth-result}
The error terms in Theorem~\ref{thm:intro-main-theorem} may be improved dramatically if the sharp cutoff $n \leq X$ is replaced by a smooth cutoff. To this effect, fix a smooth function $w(x)$ with inverse Mellin transform $W(s)$. We have
\[
	\sum_{n \geq 1} H(n) H(n-\ell) w\big(\tfrac{n}{X}\big)
	 = \frac{1}{2\pi i} \int_{(2+\epsilon)} D_\ell(s-\tfrac{1}{2}) W(s) X^s ds,
\]
provided both sides converge. If $W(s)$ decays exponentially in $\vert \Im s \vert$, we may shift the contour of integration left to $\Re s = 1 + \epsilon$ by Theorem~\ref{thm:D_h-growth}. This extracts two residues, and the shifted contour integral contributes $O((X\ell)^{1+\epsilon})$ by Theorem~\ref{thm:D_h-growth}. We conclude that
\begin{align*}
	& \sum_{n \geq 1} H(n) H(n-\ell) w\big(\tfrac{n}{X}\big) \\[-.5em]
	& \qquad \qquad
	= W(2) X^2 \Res_{s=\frac{3}{2}} D_\ell(s)
		+ W(\tfrac{3}{2}) X^\frac{3}{2} \Res_{s=1} D_\ell(s) + O_\epsilon((X\ell)^{1+\epsilon}),
\end{align*}
which offers some evidence in support of the conjecture~\eqref{eq:intro-conjecture}. \hfill //
\end{remark}

\vspace{5 mm}
\bibliographystyle{alpha}
\bibliography{compiled_bibliography.bib}

\newcommand{\etalchar}[1]{$^{#1}$}
\begin{thebibliography}{HKLDW18}

\bibitem[ASHV23]{HolgadoVicente23}
\'{A}lvaro Serrano~Holgado and Luis Manuel~Navas Vicente.
\newblock The zeta function of a recurrence sequence of arbitrary degree.
\newblock {\em Mediterranean Journal of Mathematics}, 20, 2023.

\bibitem[BF04]{BruinierFunke04}
Jan~Hendrik Bruinier and Jens Funke.
\newblock On two geometric theta lifts.
\newblock {\em Duke Mathematics Journal}, 125(1):45--90, 2004.

\bibitem[BFOR19]{BFOR19}
K.~Bringmann, A.~Folsom, K.~Ono, and L.~Rolen.
\newblock {\em {Harmonic Maass Forms and Mock Modular Forms: Theory and
  Applications}}, volume~64 of {\em Colloquium Publications}.
\newblock American Mathematical Society, 2019.

\bibitem[Coh93]{Cohen93}
Henri Cohen.
\newblock {\em {A Course in Computational Algebraic Number Theory}}.
\newblock Graduate Texts in Mathematics. Springer-Verlag, 1993.

\bibitem[DG00]{Gupta00}
Shamita Dutta~Gupta.
\newblock The {R}ankin-{S}elberg method on congruence subgroups.
\newblock {\em Illinois J. Math.}, 44(1):95--103, 2000.

\bibitem[DI83]{DeshouillersIwaniec82}
J.-M. Deshouillers and H.~Iwaniec.
\newblock Kloosterman sums and {F}ourier coefficients of cusp forms.
\newblock {\em Invent. Math.}, 70(2):219--288, 1982/83.

\bibitem[Gol76]{Goldfeld76}
D.~Goldfeld.
\newblock The class number of quadratic fields and the conjectures of {B}irch
  and {S}winnerton-{D}yer.
\newblock {\em Ann. Scuola Norm. Sup. Pisa}, \textbf{2}(4):623--663, 1976.

\bibitem[Goo81]{Good81}
A.~Good.
\newblock {Cusp Forms and Eigenfunctions of the Laplacian}.
\newblock {\em Math. Ann.}, 255:523--548, 1981.

\bibitem[GR15]{GradshteynRyzhik07}
I.~S. Gradshteyn and I.~M. Ryzhik.
\newblock {\em Table of integrals, series, and products}.
\newblock Elsevier/Academic Press, Amsterdam, eighth edition, 2015.
\newblock Translated from the Russian, Translation edited and with a preface by
  Daniel Zwillinger and Victor Moll, Revised from the seventh edition
  [MR2360010].

\bibitem[GS03]{GranvilleSoundararajan03}
A.~Granville and K.~Soundararajan.
\newblock {The Distribution of Values of $L(1,\chi_d)$}.
\newblock {\em Geometric and Functional Analysis}, 13:992--1028, 20003.

\bibitem[GZ83]{GrossZagier83}
B.~Gross and D.~Zagier.
\newblock Points de {H}eegner et deriv\'{e}es de functions {$L$}.
\newblock {\em C. R. Acad. Sci.}, \textbf{297}:85--87, 1983.

\bibitem[Hei34]{Heilbronn34}
Hans Heilbronn.
\newblock On the class number in imaginary quadratic fields.
\newblock {\em The Quarterly Journal of Mathematics}, os-5(1):150--160, 01
  1934.

\bibitem[HH16]{HoffsteinHulse13}
Jeff Hoffstein and Thomas~A. Hulse.
\newblock Multiple {D}irichlet series and shifted convolutions.
\newblock {\em J. Number Theory}, 161:457--533, 2016.
\newblock With an appendix by Andre Reznikov.

\bibitem[HKLDW18]{HKLDWSphere}
Thomas~A. Hulse, Chan~Ieong Kuan, David Lowry-Duda, and Alexander Walker.
\newblock Second moments in the generalized {G}auss circle problem.
\newblock {\em Forum Math. Sigma}, 6:Paper No. e24, 49, 2018.

\bibitem[Hux85]{Huxley85}
M.~N. Huxley.
\newblock Introduction to {K}loostermania.
\newblock In {\em Elementary and analytic theory of numbers ({W}arsaw, 1982)},
  volume~17 of {\em Banach Center Publ.}, pages 217--306. PWN, Warsaw, 1985.

\bibitem[HZ76]{HirzebruchZagier76}
F.~Hirzebruch and D.~Zagier.
\newblock Intersection numbers of curves on {H}ilbert modular surfaces and
  modular forms of {N}ebentypus.
\newblock {\em Invent. Math.}, \textbf{36}:57--113, 1976.

\bibitem[IK04]{IwaniecKowalski04}
Henryk Iwaniec and Emmanuel Kowalski.
\newblock {\em Analytic number theory}, volume~53 of {\em American Mathematical
  Society Colloquium Publications}.
\newblock American Mathematical Society, Providence, RI, 2004.

\bibitem[Iwa02]{Iwaniec02}
Henryk Iwaniec.
\newblock {\em Spectral methods of automorphic forms}, volume~53 of {\em
  Graduate Studies in Mathematics}.
\newblock American Mathematical Society, Providence, RI; Revista Matem\'atica
  Iberoamericana, Madrid, second edition, 2002.

\bibitem[Jut96]{Jutila96}
Matti Jutila.
\newblock {The additive divisor problem and its analogs for Fourier
  coefficients of cusp forms: I}.
\newblock {\em Mathematische Zeitschrift}, \textbf{223}(1):435--461, 1996.

\bibitem[Jut97]{Jutila97}
Matti Jutila.
\newblock {The additive divisor problem and its analogs for Fourier
  coefficients of cusp forms: II}.
\newblock {\em Mathematische Zeitschrift}, \textbf{225}(1):625--637, 1997.

\bibitem[K{\i}r15]{Kiral15}
Eren~Mehmet K{\i}ral.
\newblock Subconvexity for half integral weight {$L$}-functions.
\newblock {\em Math. Z.}, 281(3-4):689--722, 2015.

\bibitem[KS03]{KimSarnak03}
H.~Kim and P.~Sarnak.
\newblock Refined estimates towards the {R}amanujan and {S}elberg conjectures.
\newblock {\em J. Amer. Math. Soc.}, \textbf{16}:175--181, 2003.

\bibitem[Kum18]{Kumaraswamy18}
V.~Vinay Kumaraswamy.
\newblock On correlations between class numbers of imaginary quadratic fields.
\newblock {\em Acta Arithmetica}, 185:211--231, 2018.

\bibitem[Kuz80]{Kuznetsov80}
N.~V. Kuznetsov.
\newblock {Petersson's conjecture for cusp forms of weight zero and Linnik's
  conjecture}.
\newblock {\em Mat. Sb. (N. S.)}, \textbf{111(153)}:334--383, 1980.

\bibitem[M{\"u}l95]{Muller95}
Wolfgang M{\"u}ller.
\newblock {The Rankin--Selberg method for non-holomorphic automorphic forms}.
\newblock {\em Journal of Number Theory}, 51(1):48--86, 1995.

\bibitem[MV06]{MontgomeryVaughn06}
H.~L. Montgomery and R.~C. Vaughn.
\newblock {\em {Multiplicative Number Theory: I. Classical Theory}}.
\newblock Cambridge University Press, 2006.

\bibitem[Nel21]{Nelson21}
Paul Nelson.
\newblock The spectral decomposition of $\vert \theta \vert^2$.
\newblock {\em Math. Z.}, \textbf{298}:1425--1447, 2021.

\bibitem[ODL{\etalchar{+}}20]{DLMF}
F.~W.~J. Olver, A.~B.~Olde Daalhuis, D.~W. Lozier, B.~I. Schneider, R.~F.
  Boisvert, C.~W. Clark, B.~V.~Saunders B.~R. Mille~and, H.~S. Cohl, and eds.
  M.~A.~McClain.
\newblock {NIST Digital Library of Mathematical Functions}.
\newblock http://dlmf.nist.gov/, Release 1.0.26 of 2020-03-15, 2020.

\bibitem[Sar94]{Sarnak94}
P.~Sarnak.
\newblock Integrals of products of eigenfunctions.
\newblock {\em International Mathematics Research Notices},
  \textbf{(6)}:251--260, 1994.

\bibitem[SS77]{SerreStark77}
J.-P. Serre and H.~M. Stark.
\newblock Modular forms of weight {$1/2$}.
\newblock In {\em Modular functions of one variable, {VI} ({P}roc. {S}econd
  {I}nternat. {C}onf., {U}niv. {B}onn, {B}onn, 1976)}, Lecture Notes in Math.,
  Vol. 627, pages 27--67. Springer, Berlin, 1977.

\bibitem[SS22]{ShankhadharSingh22}
Karam~Deo Shankhadhar and Ranveer~Kumar Singh.
\newblock {An analogue of Weil's converse theorem for harmonic Maass forms of
  polynomial growth}.
\newblock {\em Research in Number Theory}, 8.36, 2022.

\bibitem[Tit86]{Titchmarsh86}
E.~C. Titchmarsh.
\newblock {\em The theory of the {R}iemann zeta-function}.
\newblock The Clarendon Press, Oxford University Press, New York, second
  edition, 1986.
\newblock Edited and with a preface by D. R. Heath-Brown.

\bibitem[VT87]{VinogradovTakhtadzhyan87}
A.~I. Vinogradov and L.~A. Takhtadzhyan.
\newblock {The Zeta Function of the Additive Divisor Problem and Spectral
  Decomposition of the Automorphic Laplacian}.
\newblock {\em Journal of Soviet Mathematics}, \textbf{36}(1):57--78, 1987.

\bibitem[Wol72]{Wolke72}
Dieter Wolke.
\newblock {Moments of Class Numbers, III}.
\newblock {\em Journal of Number Theory}, 4:523--531, 1972.

\bibitem[You19]{Young19}
Matthew~P. Young.
\newblock Explicit calculations with {E}isenstein series.
\newblock {\em Journal of Number Theory}, \textbf{199}:1--48, 2019.

\bibitem[Zag75]{Zagier75}
Don Zagier.
\newblock Nombres de classes et formes modulaires de poids 3/2.
\newblock {\em Seminaire de Théorie des Nombres de Bordeaux}, 4:1--2,
  1974-1975.

\bibitem[Zag81]{ZagierRankinSelberg}
Don Zagier.
\newblock The {R}ankin-{S}elberg method for automorphic functions which are not
  of rapid decay.
\newblock {\em J. Fac. Sci. Univ. Tokyo Sect. IA Math.}, 28(3):415--437 (1982),
  1981.

\end{thebibliography}

\end{document}